\documentclass[11pt,reqno]{amsart}
\usepackage{tikz}
\usepackage{subfig}
\usepackage{epstopdf}
\usepackage{epsfig}
\usepackage{math}
\graphicspath{{./figures/}}
\usepackage{tikz}
\usetikzlibrary{shapes,arrows}
\usepackage{adjustbox}
\usepackage{rotating}
\usepackage{graphicx}

\tikzstyle{decision} = [diamond, draw, fill=blue!20, 
    text width=4.5em, text badly centered, node distance=3cm, inner sep=0pt]
\tikzstyle{block} = [rectangle, draw, fill=blue!20, 
    text width=5em, text centered, rounded corners, minimum height=4em]
\tikzstyle{line} = [draw, -latex']
\tikzstyle{cloud} = [draw, ellipse,fill=red!20, node distance=3cm,
    minimum height=2em]
    
\usetikzlibrary{positioning}
\tikzset{main node/.style={circle,fill=blue!20,draw,minimum size=1cm,inner sep=0pt},  }

\DeclareMathOperator*{\argmin}{argmin}
\DeclareMathOperator*{\argmax}{argmax}

\newcommand{\overlinetmp}[1]{[#1]}

\begin{document}
\title[]{Computational Mean-field information dynamics associated with reaction-diffusion equations} 
\author[Li]{Wuchen Li}
\email{wuchen@mailbox.sc.edu}
\address{Department of Mathematics, University of South Carolina, Columbia }

\author[Lee]{Wonjun Lee}
\email{wlee@math.ucla.edu}
\address{Department of Mathematics, University of California, Los Angeles}

\author[Osher]{Stanley Osher}
\email{sjo@math.ucla.edu}
\address{Department of Mathematics, University of California, Los Angeles}
\newcommand{\vr}{\overrightarrow}
\newcommand{\wt}{\widetilde}
\newcommand{\dd}{\mathcal{\dagger}}
\newcommand{\ts}{\mathsf{T}}
\newcommand{\wc}[1]{{\color{blue} [li: #1]}}

\thanks{This research is supported by AFOSR MURI FA9550-18-1-0502, ONR grants N000142012093 and N000141812527.}
\begin{abstract}
We formulate and compute a class of mean-field information dynamics for reaction-diffusion equations. Given a class of nonlinear reaction-diffusion equations and entropy type Lyapunov functionals, we study their gradient flows formulations with 
generalized optimal transport metrics and mean-field control problems. {We apply the primal-dual hybrid gradient algorithm to compute the mean-field control problems with potential energies. A byproduct of the proposed method contains a new and efficient variational scheme for solving implicit in time schemes of mean-field control problems.} Several numerical examples demonstrate the solutions of mean-field control problems. 
\end{abstract}
\maketitle
\section{Introduction}
Metrics \cite{AGS} are essential in mathematical physics equations with applications in scientific computing and Bayesian sampling problems. One popular example in this field is the optimal transport metric, a.k.a. Wasserstein metric \cite{AGS, Villani2009_optimal}, defined in probability density space. The study of gradient flow in optimal transport metric space has found applications in Markov-chain-Monte-Carlo (MCMC) methods. An example is that the heat flow is the gradient descent flow of negative Boltzman-Shannon entropy in Wasserstein space \cite{otto2001}. 
In addition, the optimal control problems and differential games in Wasserstein space are known as mean-field control problems and mean-field games, respectively \cite{MFGC,LL,MRP}. Moreover, all the above variational formulations are useful in modeling a large number of particles simultaneously, which are essential in modeling inverse problems, AI inference, and optimization problems {\cite{GL,GL1,AHLS,LiY,ALOG}}. 

Recently, generalizations of optimal transport metrics and gradient flows have been studied by \cite{C1, O1, LiY, MRP}; see many references therein. They are helpful in studying nonlinear diffusion equations. Meanwhile, information geometry has been using the Fisher-Rao information metric, which can be applied to study the pure reaction equations, arisen in population games, AI inference, and modeling \cite{IG}. A ``linear'' combination of information geometry and optimal transport metrics has been applied to the study of reaction-diffusion equations \cite{AM1, AM}. One typical example is the unbalanced optimal transport metric \cite{CPSV,SV,LMS}. In this direction, \cite{GLOP, LLLO} also propose generalized optimal transport distances in unnormalized density space. 

This paper introduces numerical schemes for general metric spaces and mean-field control problems for reaction-diffusion equations. Given a Lyapunov functional (entropy) and a nonlinear reaction-diffusion equation, we study a metric space in which the given reaction-diffusion equation forms a gradient flow. In the designed metric space, we derive mean-field Hamiltonian flows and Hamilton-Jacobi equations in positive density space. They are derived from mean-field optimal control problems of nonlinear reaction-diffusion equations. { We call these equations {\em mean-field information dynamics}.} We then design primal-dual algorithms to compute the proposed dynamics. Numerically, an additional potential energy is also proposed to improve the convexity of the problem. We apply Newton's method to compute the sub-optimization problems in mean-field control problems. { Our mean-field control problem provides a variational formulation for implicit schemes of mean-field information dynamics. And the primal-dual hybrid gradient algorithm solves the implicit scheme with a simple component by component update.} The flowchart is summarized in the above figure. 
\begin{figure}
{\includegraphics[scale=0.25]{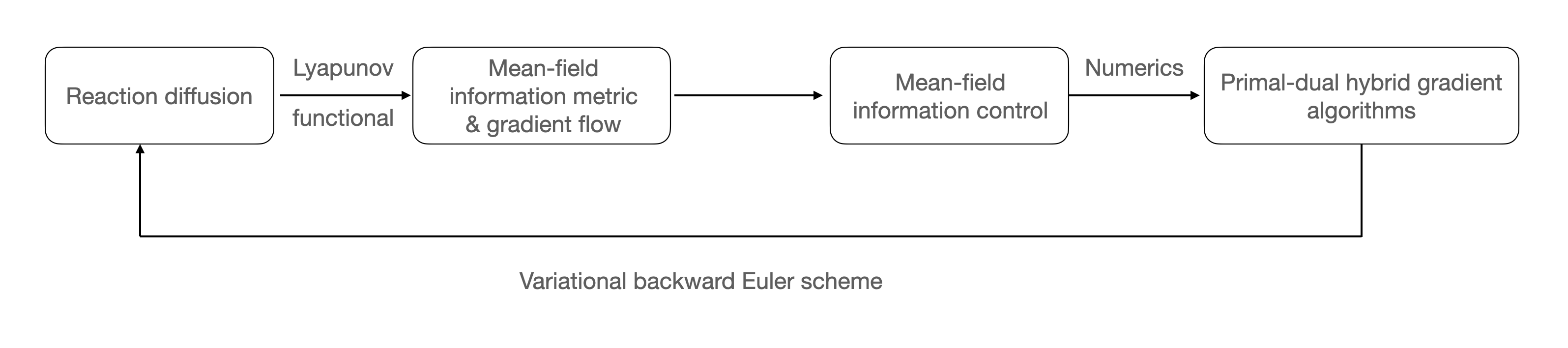}}
\label{flowchart}
\end{figure}

Various gradient flows have been studied in \cite{GZW,AM1,AM}; see many references therein. These formulations are motivated and derived from the Onsager principle. The principle is used to derive many evolution equations from soft matter physics and chemistry \cite{ON}. 
{This work focuses on both modeling and computational formalism for gradient flows, such as reaction-diffusion equations. We study mean-field control problems and generalized optimal transport metrics for reaction-diffusion equations. By using primal-dual hybrid gradient algorithms, we can efficiently compute the mean-field control problems implicitly in time. It also has a potential to compute reaction-diffusion equations implicit in time; see similar studies in \cite{LLO}.}

The paper is organized as follows. In section \ref{sec2}, we briefly review both gradient flows and Hamiltonian flows in a finite-dimensional Euclidean space. In section \ref{sec3}, we introduce the mean-field information distance in positive density space, in which we formulate both gradient flows and Hamiltonian flows in the positive density space. In section \ref{section4}, we demonstrate several concrete examples of proposed mean-field information dynamics. In section \ref{sec5}, we design the primal-dual hybrid-gradient methods to compute the mean-field information dynamics. {It also provides a scheme to compute the mean-field control problem implicitly in time.} Several numerical examples are presented. 

\section{Review}\label{sec2}
In this section, we review some facts on gradient systems and optimal control in a $d$-dimensional Euclidean space. {We next apply these facts into the infinite-dimensional space.}
\subsection{Gradient flows}
Consider a optimization problem in $\mathbb{R}^d$ 
\begin{equation*}
\min_{x\in\mathbb{R}^d}~f(x),
\end{equation*}
where the function $f\colon \mathbb{R}^d\rightarrow\mathbb{R}$ is a given smooth objective function. To find the minimizer of function $f$, consider an initial value dynamical system 
\begin{equation}\label{dynamics}
\frac{dx(t)}{dt}=-g(x(t))^{-1}\nabla f(x(t)),\quad x(0)=x_0,
\end{equation}
where $g\colon\mathbb{R}^d\rightarrow \mathbb{R}^{d\times d}$ is a given matrix function. 
In practice, there are several natural choices of matrix functions $g$. 
\begin{itemize} 
\item[(i)] If $g(x)=\mathbb{I}$, where $\mathbb{I}$ is an identity matrix. Dynamic \eqref{dynamics} forms the gradient flow in Euclidean space. 
\item[(ii)] If $g(x)=\nabla^2 f(x)$, where $\nabla^2$ is the Euclidean Hessian operator. Dynamic \eqref{dynamics} satisfies Newtonian flow in Euclidean space. 
\end{itemize}
Assume that matrix function $g$ is positive definite. We observe that the objective function $f$ decays along the 
dynamic \eqref{dynamics}. In other words, 
\begin{equation}\label{Ly}
\frac{d}{dt}f(x(t))=\nabla f(x(t))^{\ts} \frac{dx(t)}{dt}=-\nabla f(x(t))^{\ts}g(x(t))^{-1}\nabla f(x(t))\leq 0. 
\end{equation}
The above decaying behavior is known as a Lyapunov method, in which the objective function is a ``natural'' Lyapunov function for equation \eqref{dynamics}. Dynamic \eqref{dynamics} can be viewed as a gradient flow in the metric space $(\mathbb{R}^d, g)$. The matrix function $g$ is often named the metric tensor. It is also called the matrix operator or the preconditioner matrix. 
\subsection{Optimal control}
In this metric space $(\mathbb{R}^d, g)$, one often considers the following variational problem. Denote $k\colon \mathbb{R}^d\rightarrow\mathbb{R}$ as a given smooth potential function and formulate $L\colon\mathbb{R}^d\times \mathbb{R}^d\rightarrow\mathbb{R}$ as a Lagrangian function: 
\begin{equation*}
L(x,v)=\frac{1}{2}v^{\ts}g(x)v-k(x). 
\end{equation*}
Consider
\begin{equation}\label{Lagrangian}
\frac{1}{2}\mathrm{D}(x_0,x_1)^2:=\inf_{x\colon [0,1]\rightarrow \mathbb{R}^d}\int_0^1 L(x(t), \frac{dx(t)}{dt}) dt,
\end{equation}
where the infimum is taken among all smooth paths $x(t)\in\mathbb{R}^d$, $t\in[0,1]$ with fixed initial and terminal functions $x_0$, $x_1$. By direct calculations, the Euler-Lagrange equation of  problem \eqref{Lagrangian} is formulated below. {Denote a Hamiltonian
$H\colon\mathbb{R}^d\times\mathbb{R}^d\rightarrow\mathbb{R}$ as the convex conjugate of $L$.}
\begin{equation}\label{Hamiltonian}
\begin{split}
H(x,p)=\sup_{v\in \mathbb{R}^d}\quad p^{\ts}v-L(x,v)
=\frac{1}{2}p^{\ts}g(x)^{-1}p+k(x). 
\end{split}
\end{equation}
where $x\in\mathbb{R}^d$ represents the state variable and $p\in\mathbb{R}^d$ is the momentum variable. 
The minimizer of variational problem \eqref{Lagrangian} satisfies $x(0)=x_0$, $x(1)=x_1$, with 
\begin{equation}\label{PD}
\left\{\begin{aligned}
\frac{dx(t)}{dt}=&\nabla_pH(x(t),p(t)),\\
\frac{dp(t)}{dt}=&-\nabla_x H(x(t),p(t)).
\end{aligned}\right.
\end{equation}
If $k=0$, equation \eqref{PD} is called the geodesic equation in metric space $(\mathbb{R}^d, g)$, and $\mathrm{D}(x_0, x_1)$ is the distance function. The above flow defines the characteristics of Hamilton-Jacobi equation. Consider a value function $U\colon [0,\infty)\times \mathbb{R}^d\rightarrow\mathbb{R}$, such that 
\begin{equation*}
\partial_tU(t,x)+H(x, \nabla_xU(t,x))=0.
\end{equation*}
And
\begin{equation*}
p(t)=\nabla_xU(t,x). 
\end{equation*}

\subsection{Gradient flows and variational time discretizations}
{We remark that both gradient flow \eqref{dynamics} and Hamiltonian flow \eqref{PD} are different but connected with each other. One can design a variational implicit scheme for gradient flow \eqref{dynamics}.}

Formally speaking, the gradient flow \eqref{dynamics} can be written below:
\begin{equation*}
\left\{\begin{aligned}
\frac{dx(t)}{dt}=&\nabla_pH(x(t), p(t)),\\
p(t)=&-\nabla_xf(x(t)),
\end{aligned}\right.
\end{equation*}
where $H$ is the quadratic Hamiltonian function defined in \eqref{Hamiltonian} with $k=0$. {We observe that the equation $p(t)=-\nabla_xf(x(t))$ does not satisfy the second equation in \eqref{PD} directly. However, one can construct a time approximation variational scheme, which enforces $p(t)=-\nabla f(x(t))$ at the terminal time. } 

Denote a time stepsize  $\Delta t>0$. Construct a sequence $\{x_k\}_{k=1}^\infty$ below. Consider an iterative variational sequence 
\begin{equation}\label{min}
\inf_{x\colon [t_k,t_{k+1}]\rightarrow \mathbb{R}^d}~\Big\{\int_{t_k}^{t_{k+1}} L(x(t), \frac{dx(t)}{dt}) dt+f(x(t_{k+1}))\colon x(t_k)=x_k\Big\}.
\end{equation}
Write 
\begin{equation*}
x_{k+1}=x(t_{k+1}),
\end{equation*}
where $x(t_{k+1})$ is the minimizer of variational problem \eqref{min}. 
\begin{proposition}\label{prop}
The minimizer of variational problem \eqref{min} satisfies  
\begin{equation}\label{HFL1}
\left\{\begin{aligned}
&\frac{dx(t)}{dt}=\nabla_p H(x(t), p(t)), \quad t\in [t_{k}, t_{k+1})\\
&\frac{dp(t)}{dt}=-\nabla_xH(x(t), p(t)),\quad t\in [t_{k}, t_{k+1}),\\
&p(t_{k+1})=-\nabla_xf(x(t_{k+1})). 
\end{aligned}\right.
\end{equation}
Then $\{x_k\}_{k=1}^\infty$ is a first-order time discretization of gradient flow \eqref{dynamics}. 
\end{proposition}
\begin{proof}[Proof of Proposition \ref{prop}]
To see this fact, one can solve the minimization problem \eqref{min} and observe that 
\begin{equation*}
\left\{\begin{aligned}
\frac{dx}{ds}=&\nabla_pH(x,p),\quad \frac{dp}{ds}=-\nabla_x H(x,p),\\
x(t_k)=&x_k,\quad p(t_{k+1})=-\nabla_xf(x(t_{k+1})).
\end{aligned}\right.
\end{equation*}
The update forms 
\begin{equation*}
\begin{split}
x_{k+1}=&x(t_{k+1})=x_k+\int_{t_k}^{t_{k+1}}\frac{dx(t)}{dt}dt\\
=&x_k+\int_{t_k}^{t_{k+1}}\nabla_pH(x(t),p(t))dt\\
=&x_k+(t_{k+1}-t_k)\nabla_pH(x(t),p(t))|_{t=h}+o(h)\\
=&x_k-hg(x_{k+1})^{-1}\nabla_x f(x_{k+1})+o(h). 
\end{split}
\end{equation*}
{Hence $\{x_k\}_{k=1}^\infty$ is a backward Euler time discretization of $x(t)$ up to a small order perturbation. We notice that the implicit scheme is a first order time discretization in term of stepsize $h$.} 
\end{proof}

\begin{remark}
We remark that equation \eqref{PD}, \eqref{min} are consequences of Pontryagin's maximum principle. A modification of this connects with gradient flows. These modifications induce implicit schemes to approximate gradient flows. 
\end{remark}

\section{Mean-field information metric spaces and their dynamics}\label{sec3}
In this section, we first develop motivations and examples, including reaction diffusion equations and Lyapunov functionals. We next review the mean-field information metric space in positive density space; see \cite{AGS,AM1}. Finally, we formulate both gradient flows and Hamiltonian flows in positive density metric space. We define a class of mean-field control problems.
\subsection{Motivation and examples}
In this subsection, we review some known facts about reaction-diffusion equations in term of optimal transport type gradient flows. See related studies of diffusion equations in \cite{otto2001}, and reaction-diffusion equations \cite{AM1,AM}. 
In a word, we can construct metrics and mean-field control problems for some nonlinear reaction-diffusion equations from Lyapunov functionals. 

Consider a scalar nonlinear reaction-diffusion equation
\begin{equation}\label{gheat}
\partial_t u(t,x)=\Delta F(u(t,x))+R(u(t,x)), 
\end{equation}
where $x\in \Omega$, $\Omega\subset\mathbb{R}^d$ is compact convex set, $u\in \mathcal{M}(\Omega)=\{u\in C^\infty(\Omega)\colon u\geq 0\}$, and $\Delta$ is the Euclidean Laplacian operator. {We assume periodic boundary conditions on the boundary of the spatial domain $\Omega$,} and $F$, $R\colon \mathbb{R}_+\rightarrow\mathbb{R}_+$ are smooth functions. 

We next construct a Lyapunov functional $\mathcal{G}\colon \mathcal{M}(\Omega)\rightarrow\mathbb{R}$ to study equation \eqref{gheat}. Consider
\begin{equation*}
\mathcal{G}(u)=\int G(u(x))dx, 
\end{equation*}
where $G\colon \mathbb{R}\rightarrow \mathbb{R}$ is a convex function with $G''(u)>0$. In this case, along the reaction-diffusion equation \eqref{gheat}, we observe that 
\begin{equation*}
\begin{split}
\frac{d}{dt}\mathcal{G}(u(t,\cdot))=&\int G'(u(t,x))\cdot \partial_tu(t,x)dx\\
=&\int G'(u(t,x)) (\Delta F(u(t,x))+R(u(t,x))) dx\\
=&-\int \Big(\nabla G'(u(t,x)), \nabla F(u(t,x))\Big) dx+\int G'(u(t,x)) R(u(t,x))dx\\
=&-\int \Big(\nabla G'(u(t,x)), \nabla u(t,x)\Big)F'(u(t,x)) dx+\int G'(u(t,x)) R(u(t,x))dx\\
=&-\int \Big(\nabla G'(u(t,x)), \nabla u(t,x)\Big)G''(u(t,x))\frac{F'(u(t,x))}{G''(u(t,x))} dx+\int G'(u(t,x))^2 \frac{R(u(t,x))}{G'(u(t,x))}dx\\
=&-\int \Big(\nabla G'(u(t,x)), \nabla G'(u(t,x))\Big)\frac{F'(u(t,x))}{G''(u(t,x))} dx+\int G'(u(t,x))^2 \frac{R(u(t,x))}{G'(u(t,x))}dx,
\end{split}
\end{equation*}
where we apply integration by parts in the third equality and $\nabla G'(u)=G''(u)\nabla u$ in the last equality. 

We assume that $R\in C^{1}(\Omega)$ is a given function with $-\frac{R}{G'}>0$, and $F'(u)>0$ for $u>0$. Under these assumptions, it is clear that 
\begin{equation*}
\frac{d}{dt}\mathcal{G}(u)\leq 0. 
\end{equation*}
This indicates that functional $\mathcal{G}(u)$ is not increasing along flow \eqref{gheat}. 

In fact, the above decay behavior indicates a gradient flow formulation for dynamics \eqref{gheat}. We introduce the following notations. Denote an {inverse of the weighted elliptic operator} 
\begin{equation*}
g(u):=\Big(-\nabla\cdot (\frac{F'(u)}{G''(u)}\nabla)-\frac{R(u)}{G'(u)}\Big)^{-1}.
\end{equation*}
We have 
\begin{equation}\label{org}
\begin{split}
\partial_tu=&-g(u)^{-1}\frac{\delta}{\delta u}\mathcal{G}(u)\\
=&-\Big(-\nabla\cdot (\frac{F'(u)}{G''(u)}\nabla)-\frac{R(u)}{G'(u)}\Big)\frac{\delta}{\delta u}\mathcal{G}(u)\\
=&\nabla\cdot(\frac{F'(u)}{G''(u)}\nabla G'(u))+\frac{R(u)}{G'(u)}G'(u)\\
=&\Delta F(u)+ R(u),
\end{split}
\end{equation}
where $\frac{\delta}{\delta u}$ represents the $L^2$ first variation w.r.t. $u\in \mathcal{M}(\Omega)$. In the above notation, the dissipation of Lyapunov functional $\mathcal{G}$ along equation \eqref{gheat} satisfies
\begin{equation*}
\frac{d}{dt}\mathcal{G}(u)=-\int \Big(\frac{\delta}{\delta u}\mathcal{G}(u), g(u)^{-1}\frac{\delta}{\delta u}\mathcal{G}(u)\Big)dx\leq 0. 
\end{equation*}
Clearly, our assumptions on $F$, $R$ are sufficient conditions to guarantee that $g(u)$ is a ``positive definite'' operator.

\subsection{Mean-field information metrics and their gradient flows}
In this subsection, we illustrate a formal definition of metric space and gradient flows. See details in \cite{AGS,MM,AM}. 

Denote a {smooth positive density space} as
\begin{equation*}
\mathcal{M}=\Big\{u\in C^{\infty}(\Omega)\colon u> 0\Big\}.
\end{equation*}
Given $F$, $G\colon \mathbb{R}\rightarrow \mathbb{R}$ satisfying $F'(u)>0$ if $u>0$, and $G''(u)>0$. Denote 
\begin{equation*}
V_1(u)=\frac{F'(u)}{G''(u)},\qquad V_2(u)=-\frac{R(u)}{G'(u)}.
\end{equation*}  
Denote the tangent space of $\mathcal{M}$ at $u\in \mathcal{M}$ as
\begin{equation*}
  T_u\mathcal{M} = \Big\{\sigma\in C^{\infty}(\Omega)\Big\}.
\end{equation*}  
We define the $F$, $G$, $R$ induced metric in the positive density space. 

\begin{definition}[Mean-field information metric]\label{metric}
  The inner product $g(u)\colon
  {T_u}\mathcal{M}\times{T_u}\mathcal{M}\rightarrow\mathbb{R}$ is given below. For any $\sigma_1$, $\sigma_2\in T_u\mathcal{M}$, define 
  \begin{equation*} g(u)(\sigma_1,
    \sigma_2)=\int_{\Omega} \sigma_1\Big(-\nabla\cdot(V_1(u)\nabla)+V_2(u)\Big)^{-1}\sigma_2 dx,
  \end{equation*}
  where 
 \begin{equation*}
  \Big(-\nabla\cdot(V_1(u)\nabla)+V_2(u)\Big)^{-1}\colon
  {T_u}\mathcal{M}\rightarrow{T_u}\mathcal{M},
 \end{equation*}  
 denotes the inverse operator of weighted elliptic operator $-\nabla\cdot(V_1(u)\nabla)+V_2(u)$. The other formulation of metric is given below. Denote $\Phi_i\in C^{\infty}(\Omega)$, such that 
\begin{equation*}
\sigma_i=-\nabla\cdot(V_1(u)\nabla\Phi_i)+V_2(u)\Phi_i,\quad i=1,2,
\end{equation*}
Hence the metric satisfies 
\begin{equation*}\begin{split}
    g(u)(\sigma_1, \sigma_2)   =&\int_\Omega (\nabla\Phi_1, \nabla\Phi_2) V_1(u)dx+\int_\Omega \Phi_1\Phi_2V_2(u)dx.
  \end{split} 
\end{equation*}
\end{definition}
\begin{remark}
We remark that the special case of the above metric is the $L^2$-Wasserstein metric, which is well-studied in optimal transport. It corresponds to $V_1=u$, $V_2=0$. 
It also contains the Fisher-Rao metric, which is important in information geometry. It corresponds to $V_1=0$, $V_2=u$. There are several interactive studies of them in unbalanced optimal transport $V_1=V_2=u$ \cite{CPSV, LMS} and unnormalized optimal transport $V_1=u$, $V_2=1$ \cite{GLOP, LLLO}. They are different choices of metric operators $g(u)$, depending on the Lyapunov functional. We call the above metrics {\em mean-field information metrics}. 
\end{remark}
We are now ready to formulate gradient flows in $(\mathcal{M},g)$.
\begin{proposition}[Mean-field information Gradient flow]
{Given an energy functional $\mathcal{E}\colon \mathcal{M}\rightarrow \mathbb{R}$, the
gradient flow of $\mathcal{E}$ in $(\mathcal{M}(\Omega), g)$ satisfies
\begin{equation}\label{MGD}
  \partial_t u(t,x)=\nabla\cdot(V_1(u)\nabla\frac{\delta}{\delta u}\mathcal{E}(u))(t,x)-V_2(u)\frac{\delta}{\delta u}\mathcal{E}(u)(t, x).
\end{equation}
If 
\begin{equation*}
\mathcal{E}(u)=\mathcal{G}(u)=\int G(u)dx,    
\end{equation*}
then equation \eqref{MGD} forms the reaction-diffusion equation \eqref{gheat}. }
\end{proposition}
\begin{proof}
  The proof follows the definition. The gradient operator in $(\mathcal{M}(\Omega), g)$ is defined by 
  \begin{equation}\label{GD}
    g(\sigma,
    \mathrm{grad}\mathcal{E}(u))=\int \frac{\delta}{\delta u(x)}\mathcal{E}(u)\cdot\sigma(x)dx,\quad\textrm{for
      any $\sigma(x)\in T_u\mathcal{M}$}.
  \end{equation}
  In other words, 
  \begin{equation*}
     \begin{split}
      \mathrm{grad}\mathcal{E}(u)=&g(u)^{-1}\frac{\delta}{\delta u}\mathcal{F}(u)\\
      =&-\nabla\cdot(V_1(u)\nabla \frac{\delta}{\delta u}\mathcal{E}(u))+V_2(u)\frac{\delta}{\delta u}\mathcal{E}(u),
\end{split}     
  \end{equation*} 
  which finishes the proof. Thus the gradient flow in $(\mathcal{M}(\Omega), g)$ satisfies
  \begin{equation*}
    \partial_tu(t,x)=-\mathrm{grad}\mathcal{E}(u)(t,x)=\nabla\cdot(V_1(u)\nabla \frac{\delta}{\delta u}\mathcal{E}(u))-V_2(u)\frac{\delta}{\delta u}\mathcal{E}(u).
  \end{equation*}
If $\mathcal{E}(u)=\int_\Omega G(u) dx$, then
\begin{equation*}
\begin{split}
\partial_t u=&\nabla\cdot(V_1(u)\nabla \frac{\delta}{\delta u}\mathcal{E}(u))-V_2(u)\frac{\delta}{\delta u}\mathcal{E}(u)\\
=&\nabla\cdot(\frac{F'}{G''}\nabla G')+G'\cdot\frac{R}{G'}\\
=&\nabla\cdot(\frac{F'}{G''} G''\nabla u)+G'\cdot\frac{R}{G'}\\
=&\Delta F(u)+R(u).
\end{split}
\end{equation*}
\end{proof}
We next present the decay of the Lyapunov functional along gradient flow equation \eqref{gheat}.  
\begin{proposition}[Mean-field information De-Bruijn identity]
  Suppose $u(t,x)$ satisfies \eqref{gheat}, then
  \begin{equation*}
    \frac{d}{dt}\mathcal{G}(u)=-\mathcal{I}(u),
  \end{equation*}
  where $\mathcal{I}\colon \mathcal{M}(\Omega)\rightarrow\mathbb{R}$ is a functional:
  \begin{equation}\label{GF}
  \begin{split}
    \mathcal{I}(u)=&\int_\Omega \|\nabla G'(u)\|^2V_1(u)dx+\int_\Omega |G'(u)|^2V_2(u)dx.
\end{split}
\end{equation}
\end{proposition}
\begin{proof}
The proof follows from the definition of gradient flow. Note that along the gradient flow \eqref{gheat}, 
\begin{equation*}
\begin{split}
\frac{d}{dt}\mathcal{G}(u)=& \int_\Omega (\frac{\delta}{\delta u}\mathcal{G}(u), \partial_t u) dx\\
=&\int_\Omega (G'(u), \nabla\cdot(V_1(u)\nabla G'(u))-V_2(u)G'(u))dx\\
=&-\int_\Omega \|\nabla G'(u)\|^2 V_1(u) dx-\int_\Omega |G'(u)|^2V_2(u) dx\\
=&-\mathcal{I}(u),
\end{split}
\end{equation*}
where the third equality holds following the integration by parts formula. 
\end{proof}
\begin{remark}
We remark that if $G=u\log u$, $V_1=u$, $V_2=0$, then the gradient flow satisfies the heat equation. And the decay of Lyapunov functional along the heat flow satisfies 
\begin{equation*}
\mathcal{I}(u)=\int_\Omega \|\nabla \log u\|^2u dx. 
\end{equation*}
In literature, the relation $\frac{d}{dt}\mathcal{G}(u)=-\mathcal{I}(u)$ is often named the De-Bruijn identity. And $\mathcal{I}(u)$ is called the Fisher information functional. Following this spirit, we name the generalized dissipation property
``mean-field information De-Bruijn identity''. And we call $\mathcal{I}$ the ``mean-field information functional''; see examples in \cite{LiY}. 
\end{remark}
\begin{remark}
The gradient flow not only works for a scalar function $u$. One can define a similar metric operator for a vector valued function $u$; see examples in \cite{AM}. It is worth mentioning that there are more general choices of $V_1$, which includes kernel functions;
see examples in \cite{C1,O1,LiHess}. 
\end{remark}
\begin{remark}
In information geometry \cite{IG} and its applications in machine learning, the Fisher-Rao gradient flow is known as the natural gradient flow. The Fisher-Rao metric refers to $V_1(u)=0$, $V_2(u)=u$. This metric and gradient flow has been widely used in machine learning. In addition, the mean-field information gradient flow is the generalization of the ``natural gradient'' flow. The terminology ``natural'' corresponds to the ``projection'' operation. In other words, one projects the infinite dimensional metric space into finite dimensional parameterized models, e.g. neural networks. In this paper, we focus on the infinite dimensional gradient flows, and design classical finite volume methods to solve the related dynamics. We postpone the related AI scientific computing methods in future work. See an initial approach in \cite{ALOG}. 
\end{remark}

\subsection{Mean-field information control problems}
In this subsection, we state the main variational problem studied in this paper, for which we will design fast numerical methods. We first study the critical point of a variational problem in positive density space. {We call the derived system {\em mean-field information dynamics}.}

\begin{definition}[Mean-field information control problems]
Denote an energy functional  $\mathcal{F}\colon \mathcal{M}(\Omega)\rightarrow\mathbb{R}$, and write 
\begin{equation*}
V_1(u)=\frac{F'(u)G'(u)}{G''(u)},\qquad V_2(u)=-\frac{R(u)}{G'(u)}.    
\end{equation*}
Consider a variational problem  
\begin{subequations}\label{osher}
\begin{equation}\label{oshera}
\begin{split}
&\inf_{v_1, v_2, u}\quad\int_0^1\Big[\int_\Omega \frac{1}{2}\|v_1(t,x)\|^2 V_1(u(t,x))+ \frac{1}{2}|v_2(t,x)|^2V_2(u(t,x))dx-\mathcal{F}(u)\Big]dt,
\end{split}
\end{equation}
where the infimum is taken among all density functions $u\colon [0,1]\times\Omega\rightarrow\mathbb{R}$, vector fields $v_1\colon [0,1]\times \Omega\rightarrow\mathbb{R}^d$, and reaction rate functions $v_2\colon [0,1]\times \Omega\rightarrow\mathbb{R}$, such that 
\begin{equation}\label{osherb}
\partial_t u(t,x) + \nabla\cdot( V_1(u(t,x)) v_1(t,x))=v_2(t,x)V_2(u(t,x)),
\end{equation}
with fixed initial and terminal density functions $u_0$, $u_1\in\mathcal{M}(\Omega)$.

\end{subequations}
\end{definition}
{ We briefly explain variational problem \eqref{osher} with a modeling perspective. It is a generalized optimal control problem in optimal transport \cite{BB,Villani2009_optimal} and mean-field control \cite{LL,MRP}. Suppose an infinite number of identical particles/agents evolve under both transportation and reaction. The transportation mobility is selected as $V_1$, and the reaction mobility is chosen as $V_2$. Suppose that the mean-field limit of these particles exits, which satisfies an unnormalized density function. And the evolution of density function satisfies equation \eqref{osherb}. 
Given two sets of densities $u_0$, $u_1$, what is the optimal way to move or control density $u_0$ to density $u_1$? The ``optimal'' is in the sense of the objective functional, which combines transportation and reaction kinetic energies with a potential energy. We notice that the consideration of general reaction mobility functions $V_2$ has not been considered in mean-field control/game communities \cite{LL,MRP}. We expect that the proposed variational problems will be useful in controlling reaction-diffusion models, which arise in biology, chemistry, and,  recently, social dynamics and pandemic evolution.}

We next obtain the critical point for the variational problem \eqref{osher}. Assume that a minimizer for variational problem \eqref{osher} exists. We formally present the derivation of the critical point.  

\begin{proposition}[Mean-field information Hamiltonian flows]\label{MFH}
Assume $u(t,x)>0$ for $t\in[0,1]$.
Then there exists a function $\Phi\colon [0,1]\times\Omega\rightarrow\mathbb{R}$, such that the critical points of variational problem \eqref{osher} satisfy
\begin{equation*}
v_1(t,x)=\nabla\Phi(t,x),\quad v_2(t,x)=\Phi(t,x), 
\end{equation*}
with 
\begin{equation}\label{mins}
\left\{\begin{aligned}
&\partial_tu(t,x) +\nabla\cdot(V_1(u(t,x))\nabla\Phi(t,x))= V_2(u(t,x))\Phi(t,x),\\
&\partial_t\Phi(t,x)+\frac{1}{2}\|\nabla\Phi(t,x)\|^2 V_1'(u(t,x))+\frac{1}{2}|\Phi(t,x)|^2V'_2(u(t,x))+\frac{\delta}{\delta u}\mathcal{F}(u)(t,x)=0,
\end{aligned}\right.
\end{equation} 
and 
\begin{equation*}
u(0,x)=u_0(x),\qquad u(1,x)=u_1(x).     
\end{equation*}
\end{proposition}
\begin{proof}
We first rewrite the variables in variational formula \eqref{osher} as
\begin{equation*}
m_1(t,x)=V_1(u)v(t,x),\quad m_2(t,x)=V_2(u)v_2(t,x), 
\end{equation*}
 Then variational problem \eqref{osher} forms 
\begin{equation}\label{variation}
\begin{split}
&\inf_{m_1, m_2, u} \Big\{\int_0^1 \int_\Omega \frac{\|m_1(t,x)\|^2}{2V_1(u(t,x))}+\frac{|m_2(t,x)|^2}{2V_2(u(t,x))}-\mathcal{F}(u)dxdt \colon \\
&\hspace{2cm} \partial_t u(t,x)+ \nabla\cdot m_1(t,x) =m_2(t,x), \quad \textrm{fixed $u_0$, $u_1$}\Big\}.
\end{split}
\end{equation}
Denote the Lagrange multiplier of problem \eqref{variation} by $\Phi$. We consider the following saddle point problem
\begin{equation*}
\begin{split}
\inf_{m_1, m_2, u}\sup_\Phi \quad \mathcal{L}(m_1, m_2, u,\Phi),
\end{split}
\end{equation*}
with 
\begin{equation*}
\begin{split}
\mathcal{L}(m_1, m_2,u,\Phi)= &\int_0^1 \int_\Omega \Big\{ \frac{\|m_1(t,x)\|^2}{2V_1(u(t,x))}+\frac{|m_2(t,x)|^2}{2V_2(u(t,x))} \\
&\hspace{1cm}+\Phi(t,x)\Big(\partial_t u(t,x)+ \nabla\cdot m_1(t,x)-m_2(t,x)\Big)\Big\} dx dt.
\end{split}
\end{equation*}
By finding the saddle point of $\mathcal{L}$,  we have
\begin{equation*}
\left\{\begin{split}
&\frac{\delta}{\delta m_1}\mathcal{L}=0,\\
& \frac{\delta}{\delta m_2}\mathcal{L}=0,\\
 &\frac{\delta}{\delta u}\mathcal{L}=0,\\
 &\frac{\delta}{\delta \Phi}\mathcal{L}=0,
\end{split}\right.\quad\Rightarrow\quad\left\{\begin{split}
&\frac{m_1}{V_1}=\nabla\Phi,\\
& \frac{m_2}{V_2}=\Phi,\\
 &-\frac{1}{2}\frac{\|m_1\|^2}{V_1^2}V_1'-\frac{1}{2}\frac{|m_2|^2}{V_2^2}V_2'-\frac{\delta}{\delta u}\mathcal{F}-\partial_t\Phi=0,\\
 &\partial_tu+\nabla\cdot m_1-m_2=0,
\end{split}\right.
\end{equation*}
where $\frac{\delta}{\delta m_1}$, $\frac{\delta}{\delta m_2}$, $\frac{\delta}{\delta u}$, $\frac{\delta}{\delta\Phi}$ are $L^2$ first variations w.r.t. functions $m_1$, $m_2$, $u$, $\Phi$, respectively. Substituting the above two row equations into the last two row equations, we derive the PDE pair \eqref{mins} in $\mathcal{M}(\Omega)$. 
\end{proof}
\begin{remark}
If $V_1=u$, $V_2=0$, the above formulation corresponds to the well-known Benamou-Brenier formula \cite{BB} in optimal transport. 
\end{remark}
\begin{remark}
If $V_1$, $V_2$ are positive functions and are convex w.r.t. $u$, and the functional $\mathcal{F}$ is convex w.r.t. $u$, then the objective functional of problem \eqref{variation} is convex. In this case, the derived flow is a minimizer of variational problem \eqref{variation}. 
\end{remark}
\begin{proposition}[Functional Hamilton-Jacobi equations in positive density space]
The Hamilton-Jacobi equation in positive density space satisfies
\begin{equation*}
\partial_t\mathcal{U}(t,u)+\frac{1}{2}\int_\Omega \|\nabla \frac{\delta}{\delta u(x)}\mathcal{U}(t,u)\|^2V_1(u)dx+\frac{1}{2}\int_\Omega  |\frac{\delta}{\delta u(x)}\mathcal{U}(t,u)|^2V_2(u)dx+\mathcal{F}(u)=0,
\end{equation*}
where $\mathcal{U}\colon [0,1]\times L^2(\Omega)\rightarrow\mathbb{R}$ is a value functional.  
\end{proposition}
\begin{proof}
Denote the Hamiltonian functional $\mathcal{H}\colon L^2(\Omega)\times L^2(\Omega)\rightarrow\mathbb{R}$ as 
\begin{equation}\label{H}
\mathcal{H}(u,\Phi)=\int_\Omega \Big(\frac{1}{2}\|\nabla\Phi\|^2V_1(u)+\frac{1}{2}|\Phi|^2V_2(u)\Big)dx+\mathcal{F}(u). 
\end{equation}
Then the minimizer system \eqref{mins} satisfies 
\begin{equation*}
\partial_t u=\frac{\delta}{\delta \Phi}\mathcal{H}(u, \Phi),\quad \partial_t\Phi=-\frac{\delta}{\delta u}\mathcal{H}(u,\Phi),\quad \textrm{fixed $u_0$,~$u_1$}.
\end{equation*}
Here the density function $u$ is the state variable, while potential function $\Phi$ is the momentum variable in positive density space. The above flow forms the characteristic equation of Hamilton-Jacobi equation in positive density space. This is true by the fact that 
\begin{equation*}
    \frac{\delta}{\delta u(x)}\mathcal{U}(t,u)=\Phi(t,x).
\end{equation*}
\end{proof}

\subsection{Variational time discretization}
In this section, {we present a mean-field control problem, which 
gives a first order accuracy in small time interval limit for the reaction-diffusion equation \eqref{gheat}.}

From now on, we consider   
\begin{equation}\label{MFH}
    \mathcal{H}(u,\Phi)=\frac{1}{2}\int_\Omega\Big[ \|\nabla\Phi\|^2V_1(u)+|\Phi|^2V_2(u)\Big] dx.
\end{equation}
\begin{proposition}\label{prop8}
{The reaction-diffusion equation \eqref{gheat} can be formulated as}
\begin{equation*}
\left\{\begin{split}
\partial_t u(t,x)=&\frac{\delta}{\delta\Phi(x)}\mathcal{H}(u, \Phi)=-\nabla\cdot(V_1(u(t,x))\nabla \Phi(t,x))+V_2(u(t,x))\Phi(t,x),\\
\Phi(t,x)=&-\frac{\delta}{\delta u(x)}\mathcal{G}(u)=-G'(u(t,x)).
\end{split}\right.
\end{equation*}
\end{proposition}
\begin{proof}
{The proof is based on a direct calculation. Notice 
\begin{equation*}
\frac{\delta}{\delta\Phi}\mathcal{H}(u,\Phi)=-\nabla\cdot(V_1\nabla\Phi)+V_2\Phi.    
\end{equation*}
Hence 
\begin{equation*}
\begin{split}
\frac{\delta}{\delta\Phi}\mathcal{H}(u,\Phi)|_{\Phi=-G'}
=&\nabla\cdot(V_1\nabla G')-V_2G'\\
=&\nabla\cdot(V_1\frac{F'G''}{V_1}\nabla u)+\frac{R}{G'}G'\\
=&\Delta F+R. 
\end{split}
\end{equation*}
The second equality follows from the definition of $V_1$, $V_2$. 
}
\end{proof}

We notice that proposition \ref{prop8} is useful in designing a variational implicit time discretization. It is the mean-field control generalization of Jordan-Kinderlehrer-Otto (JKO) scheme \cite{CCWW, JKO}, where they select  $V_1(u)=u$, $V_2(u)=0$, and $\mathcal{F}(u)=0$. Similarly, we consider an iterative sequence of variational problems, which approximates equation \eqref{gheat} sequentially, in each time interval $[t_k, t_{k+1}]$. 
\begin{definition}[Iterative variational formulations for reaction-diffusion equations]
Denote a time stepsize as $\Delta t>0$, and $t_k=k\Delta t$, $k=0,1,2,\cdots$, and $u_0(x)=u(0,x)$. Consider the following iterative variational problem
\begin{equation}\label{osher1}
\begin{split}
&\inf_{v_1, v_2, u(\cdot,\cdot),u_{k+1}}\int_{t_k}^{t_{k+1}}\Big\{\int_\Omega \frac{1}{2}\|v_1(t,x)\|^2 V_1(u(t,x))+ \frac{1}{2}|v_2(t,x)|^2V_2(u(t,x))dx\Big\}dt\\
&\hspace{3cm}+{\mathcal{G}(u(t_{k+1},\cdot))},
\end{split}
\end{equation}
where the infimum is taken among all density functions $u\colon [t_k, t_{k+1}]\times\Omega\rightarrow\mathbb{R}$, vector fields $v_1\colon [t_k,t_{k+1}]\times \Omega\rightarrow\mathbb{R}^d$, and reaction functions $v_2\colon [t_k,t_{k+1}]\times \Omega\rightarrow\mathbb{R}$, such that 
\begin{equation*}
\partial_t u(t,x) + \nabla\cdot( V_1(u(t,x)) v_1(t,x))=v_2(t,x)V_2(u(t,x)), 
\end{equation*}
with a fixed initial value function $u(t_k,x)=u_k(x)$ and a terminal energy functional $\mathcal{G}(u_{k+1})$. Denote the update as  
\begin{equation*}
u_{k+1}(x)=u(t_{k+1},x),\quad k=1,2,3,\cdots  
\end{equation*}
where $u(t_{k+1}, x)$ is the minimizer for variational problem \eqref{osher1}. 
\end{definition}
\begin{proposition}
{Consider the minimizer of variational problem \eqref{osher1}:
\begin{equation}\label{HFL}
\left\{\begin{aligned}
&\partial_t u+\nabla\cdot(V_1\nabla\Phi)=V_2\Phi,\hspace{1.8cm} t\in[t_k, t_{k+1}),\\
&\partial_t \Phi+\frac{1}{2}\|\nabla\Phi\|^2V_1'+|\Phi|^2V_2'=0,\qquad t\in [t_k, t_{k+1}),\\
&u(t_k,x)=u_k(x),\quad \Phi(t_{k+1},x)=-\frac{\delta}{\delta u(x)}\mathcal{G}(u)|_{t=t_{k+1}}=-G'(u_{k+1}). 
\end{aligned}\right.
\end{equation}}
Then $\{u_k\}_{k=1}^{\infty}$ in \eqref{HFL} approximates the reaction-diffusion equation \eqref{gheat} with the first order accuracy in time.
\end{proposition}

\begin{proof}
 As the proof in Proposition \ref{MFH}, we derive the minimizer system for variational problem \eqref{osher1}. { We note that the minimizer system follows the Pontryagin maximum principle.}
Again, denote the Lagrangian multiplier of problem \eqref{osher1} as $\Phi$. We consider the following saddle point problem
\begin{equation*}
\begin{split}
\inf_{m_1, m_2, u(t,\cdot), u_{k+1}}\sup_\Phi \quad \mathcal{L}_1(m_1, m_2, u,\Phi),
\end{split}
\end{equation*}
where
\begin{equation*}
\begin{split}
\mathcal{L}_1(m_1, m_2,u,\Phi, u_{k+1})= &\int_{t_k}^{t_{k+1}} \int_\Omega \Big\{ \frac{\|m_1(t,x)\|^2}{2V_1(u(t,x))}+\frac{|m_2(t,x)|^2}{2V_2(u(t,x))} \\
&\hspace{1.5cm}+\Phi(t,x)\Big(\partial_t u(t,x)+ \nabla\cdot m_1(t,x)-m_2(t,x)\Big)\Big\} dxdt+\mathcal{G}(u_{k+1}).
\end{split}
\end{equation*}
Similarly, by finding the saddle point of $\mathcal{L}_1$, we have
\begin{equation*}
\left\{\begin{split}
&\frac{\delta}{\delta m_1}\mathcal{L}_1=0,\quad \frac{\delta}{\delta m_2}\mathcal{L}_1=0,\\
 &\frac{\delta}{\delta u}\mathcal{L}_1=0,\\
 &\frac{\delta}{\delta \Phi}\mathcal{L}_1=0,\\
 &\frac{\delta}{\delta u_{k+1}}\mathcal{G}(u_{k+1})=0,
\end{split}\right.\quad\Rightarrow\quad\left\{\begin{split}
&\frac{m_1}{V_1}=\nabla\Phi,\quad\frac{m_2}{V_2}=\Phi,\\
 &-\frac{1}{2}\frac{\|m_1\|^2}{V_1^2}V_1'-\frac{1}{2}\frac{|m_2|^2}{V_2^2}V_2'-\partial_t\Phi=0,\\
 &\partial_tu+\nabla\cdot m_1-m_2=0,\\
 &{\Phi(t_{k+1}, x)+\frac{\delta}{\delta u_{k+1}(x)}\mathcal{G}(u_{k+1})=0}.
\end{split}\right.
\end{equation*}
In other words, we have
\begin{equation*}
\left\{\begin{aligned}
\partial_t u=&\frac{\delta}{\delta \Phi}\mathcal{H}(u, \Phi),\quad \partial_t\Phi=-\frac{\delta}{\delta u}\mathcal{H}(u,\Phi),\\
 u_0(x)=&u_k(x),\qquad\Phi(t_{k+1},x)=-\frac{\delta}{\delta u_{k+1}(x)}\mathcal{G}(u_{k+1}),
\end{aligned}\right.
\end{equation*}
where 
\begin{equation*}
\mathcal{H}(u,\Phi)=\frac{1}{2}\int_\Omega \Big[\|\nabla\Phi\|^2V_1(u)+|\Phi|^2V_2(u) \Big]dx.     
\end{equation*}
We notice that the sequence $\{u_{k}\}_{k=1}^\infty$ forms an approximation for reaction-diffusion \eqref{gheat}:
\begin{equation*}
\begin{split}
u_{k+1}(x)=&u_0(x)+\int_{t_{k}}^{t_{k+1}} \partial_t u(t,x )dt\\
=&u_k(x)+\int_{t_k}^{t_{k+1}} \frac{\delta}{\delta \Phi(x)}\mathcal{H}(u(t,x), \Phi(t,x))dt\\
=&u_k(x)+ (t_{k+1}-t_k)\cdot \frac{\delta}{\delta \Phi(x)}\mathcal{H}(u, \Phi)|_{\Phi=\Phi(t_{k+1},x)}+o(\Delta t)\\
=&u_k(x)+ \Delta t\cdot\frac{\delta}{\delta \Phi(x)}\mathcal{H}(u, \Phi)|_{\Phi=-\frac{\delta}{\delta u_{k+1}}\mathcal{G}(u_{k+1})}+o(\Delta t).
\end{split}
\end{equation*}
The above update is a time discretization for equation \eqref{gheat}, which is true for a small order time increment $\Delta t$. 
We finish the derivation. 
\end{proof}
{
\begin{remark}
We remark that equation \eqref{mins}, \eqref{HFL} are again consequences of Pontryagin maximum principles. They are generalizations of equations in optimal transport and mean field control/game problems.
\end{remark}
\begin{remark}
We note that equation \eqref{HFL} is different from equation \eqref{gheat}. 
In gradient flow \eqref{gheat}, $\Phi$ is chosen as the $L^2$ gradient of Lyapunov functional $\mathcal{G}$, while in \eqref{HFL}, $\Phi$ satisfies a dual equation. However, if we intentionally ``ingore'' the equation \eqref{HFL} of $\Phi$ and keep solving the equation \eqref{HFL} of $u$, we obtain a time approximation scheme for reaction-diffusion equations. In this way, we let the terminal condition $\Phi(t_{k+1})=-G'(u(t_{k+1}))$ enter the system. And the first equation of system \eqref{HFL} does approximate the original reaction-diffusion equation. 
\end{remark}

\begin{remark}
If there is a Lyapunov functional $\mathcal{G}$ and functional $\mathcal{H}$, such that variational problem \eqref{osher1} becomes a convex optimization. We can develop a convex optimization method to approximate reaction-diffusion equation implicitly in time.  We leave these careful studies and computations for future work. In the numerical section of this paper, we develop a new and efficient algorithm for solving problem \eqref{osher}.   
\end{remark}
\begin{remark}
Variational problem \eqref{osher1} can be viewed as a generalized Moreau envelope problem in $(\mathcal{M}, g)$. Consider
\begin{equation*}
\begin{split}
u_{k+1}=&\arg\inf_{u\in\mathcal{M}}\quad\frac{1}{2h}\mathrm{Dist}(u_k, u)^2+ \mathcal{G}(u). 
\end{split}
\end{equation*}
Here the distance functional $\mathrm{Dist}(u_k, u)^2$ is the value function in variational problem \eqref{osher}, where we select $\mathcal{F}=0$. In detail, 
\begin{equation*}
    \mathrm{Dist}(u_k, u)^2:=\inf_{v_1, v_2, u}\quad\int_0^1\Big[\int_\Omega \|v_1(t,x)\|^2 V_1(u(t,x))+ |v_2(t,x)|^2V_2(u(t,x))\Big]dxdt, 
\end{equation*}
where the infimum is taken among $u$, $v_1$, $v_2$, such that
\begin{equation*}
\partial_t u(t,x) + \nabla\cdot( V_1(u(t,x)) v_1(t,x))=v_2(t,x)V_2(u(t,x)),
\end{equation*}
with fixed initial and terminal density functions $u_0$, $u_1$. 
\end{remark}}
\begin{remark}
We remark that the implicit time discretizations of gradient flows are not unique. There are many other semi-implicit variational time discretization methods, e.g. Crank–Nicolson algorithm  \cite{CCWW,GZW}. 
\end{remark}
\begin{remark}It is worth mentioning that the variational scheme could be useful in computing ``nonlinear'' reaction-diffusion equations, when there exists a non-quadratic functional $\mathcal{H}$. See examples in \cite{AGS,MRP}.
\end{remark}

\section{Examples}\label{section4}
In this section, we list several examples. They are designed by using both Lyapunov functionals and reaction-diffusion equations. From now on, we also study an additional energy functional $\mathcal{F}$ for the mean-field control problem \eqref{variation}. We shall design numerical schemes for them using primal-dual hybrid gradient methods. 
\begin{example}[Wasserstein metric and heat flow]
  Let 
 \begin{equation*}
  G(u)=u\log u-1, \quad F(u)=u,\quad R(u)=0,
 \end{equation*}
 thus 
  \begin{equation*}
    V_1(u)=\frac{F'(u)}{G''(u)}=u, \quad V_2(u)=-\frac{R(u)}{G'(u)}=0. 
  \end{equation*}
 The metric forms 
  \begin{equation*}
g(u)(\sigma_1,\sigma_2)=\int_\Omega (\nabla \Phi_1(x),\nabla \Phi_2(x))u(x) dx,
  \end{equation*}
  with $\sigma_i=-\nabla\cdot(u\nabla\Phi_i)$, $i=1,2$. In this case, the mean-field information metric coincides with the Wasserstein-2 metric \cite{AGS,GNT,otto2001, Villani2009_optimal}. The gradient flow of $\mathcal{G}(u)$, named negative Boltzmann-Shannon entropy, in $(\mathcal{M}(\Omega), g)$ forms the heat equation, i.e. 
  \begin{equation*}
  \partial_tu=\nabla\cdot(u\nabla G'(u))=\nabla\cdot(u G''(u)\nabla u)=\Delta u. 
  \end{equation*} 
The dissipation of $\mathcal{G}(u)$ forms 
  \begin{equation*}
  \mathcal{I}(u)=\int_\Omega \|\nabla \log u(x)\|^2 u(x) dx.
  \end{equation*}
And the Hamilton-Jacobi equation in $(\mathcal{M}(\Omega), g)$ follows  
\begin{equation*}
\partial_t\mathcal{U}(t,u)+\frac{1}{2}\int_\Omega \|\nabla \frac{\delta}{\delta u(x)}\mathcal{U}(t,u)\|^2u(x)dx+\mathcal{F}(u)=0.
\end{equation*}
Its ``characteristics'' in $(\mathcal{M}(\Omega), g)$ satisfy 
\begin{equation*}
\left\{\begin{aligned}
&\partial_tu+\nabla\cdot(u\nabla\Phi)=0,\\
&\partial_t\Phi+\frac{1}{2}\|\nabla\Phi\|^2+\frac{\delta}{\delta u}\mathcal{F}(u)=0.
\end{aligned}\right.
\end{equation*}
\end{example}

\begin{example}[Generalized Wasserstein metric and nonlinear heat flow]
Choose functions $F$, $G$, $R$, such that 
\begin{equation*}
\frac{F'(u)}{G''(u)}=u^\alpha,\quad R(u)=0, 
\end{equation*}
where $\alpha\in\mathbb{R}$. 
The metric forms 
  \begin{equation*}
g(u)(\sigma_1,\sigma_2)=\int_\Omega (\nabla \Phi_1(x),\nabla \Phi_2(x))u^\alpha(x) dx,
  \end{equation*}
  with $\sigma_i=-\nabla\cdot(u^\alpha\nabla\Phi_i)$, $i=1,2$. The gradient flow of $\mathcal{G}$ in $(\mathcal{M}(\Omega), g)$ forms 
  \begin{equation*}
  \partial_t u=\nabla\cdot(V_1(u)\nabla G'(u))=\nabla\cdot(\frac{F'(u)}{G''(u)}G''(u)\nabla u)=\nabla\cdot(F'(u)\nabla u)=\Delta F(u), 
  \end{equation*}
  and the dissipation of $\mathcal{G}(u)$ satisfies  
  \begin{equation*}
  \mathcal{I}(u)=\int_\Omega \|\nabla G'(u)\|^2u^\alpha dx. 
  \end{equation*}
And the Hamilton-Jacobi equation in $(\mathcal{M}(\Omega), g)$ follows 
\begin{equation*}
\partial_t\mathcal{U}(t,u)+\frac{1}{2}\int_\Omega \|\nabla \frac{\delta}{\delta u(x)}\mathcal{U}(t,u)\|^2u^\alpha(x)dx+\mathcal{F}(u)=0.
\end{equation*}
Its ``characteristics'' in $(\mathcal{M}(\Omega), g)$ satisfy 
\begin{equation*}
\left\{\begin{aligned}
&\partial_tu+\nabla\cdot(u^\alpha\nabla\Phi)=0,\\
&\partial_t\Phi+\frac{\alpha}{2}\|\nabla\Phi\|^2u^{\alpha-1}+\frac{\delta}{\delta u}\mathcal{F}(u)=0.
\end{aligned}\right.
\end{equation*}

\end{example}
\begin{example}[$H^{-1}$ metric and nonlinear heat flow]
Consider {$\alpha=0$} in the above example. We choose functions $F$, $G$, $R$, such that 
\begin{equation*}
\frac{F'(u)}{G''(u)}=1,\quad R(u)=0. 
\end{equation*}
In this case, 
\begin{equation*}
V_1(u)=1,\quad V_2(u)=0. 
\end{equation*}
The metric forms 
  \begin{equation*}
g(u)(\sigma_1,\sigma_2)=\int_\Omega (\nabla \Phi_1(x),\nabla \Phi_2(x))dx,
  \end{equation*}
  with $\sigma_i=-\nabla\cdot(\nabla\Phi_i)$, $i=1,2$. The gradient flow of $\mathcal{G}$ in $(\mathcal{M}(\Omega), g)$ satisfies 
  \begin{equation*}
  \partial_t u=\Delta F(u), 
  \end{equation*}
  and the dissipation of $\mathcal{G}(u)$ satisfies  
  \begin{equation*}
  \mathcal{I}(u)=\int_\Omega \|\nabla G'(u)\|^2 dx. 
  \end{equation*}
And the Hamilton-Jacobi equation in $(\mathcal{M}(\Omega), g)$ follows 
\begin{equation*}
\partial_t\mathcal{U}(t,u)+\frac{1}{2}\int_\Omega \|\nabla \frac{\delta}{\delta u(x)}\mathcal{U}(t,u)\|^2dx+\mathcal{F}(u)=0.
\end{equation*}
Its ``characteristics'' in $(\mathcal{M}(\Omega), g)$ satisfy 
\begin{equation*}
\left\{\begin{aligned}
&\partial_tu+\nabla\cdot(\nabla\Phi)=0,\\
&\partial_t\Phi+\frac{\delta}{\delta u}\mathcal{F}(u)=0.
\end{aligned}\right.
\end{equation*}
\end{example}

 \begin{example}[Fisher-Rao metric and birth-death equation]
 Consider 
 \begin{equation*}
 F(u)=0,\quad G(u)=u\log u-u,\quad R(u)=-u\log u, 
 \end{equation*} 
then 
\begin{equation*}
V_1(u)=\frac{F'(u)}{G''(u)}=0, \quad V_2(u)=-\frac{R(u)}{G'(u)}=u. 
\end{equation*}
The metric satisfies
 \begin{equation*}
 g(u)(\sigma_1,\sigma_2)=\int_\Omega \Phi_1(x)\Phi_2(x)u(x)dx, 
 \end{equation*}
with $\sigma_i(x)=\Phi_i u$, $i=1,2$. 
In this case, the mean-field information metric forms the Fisher-Rao metric in positive density space; see information geometry \cite{IG}. The gradient flow of $\mathcal{G}$ in $(\mathcal{M}(\Omega), g)$ satisfies the birth-death dynamics 
 \begin{equation*}
 \partial_t u=-V_2(u)G'(u)=-V_2(u)\log u=-u\log u. 
 \end{equation*} 
 And the dissipation of $\mathcal{G}(u)$ forms 
\begin{equation*}
\mathcal{I}(u)=\int_\Omega |\log u(x)|^2u(x)dx. 
\end{equation*}
And the Hamilton-Jacobi equation in $(\mathcal{M}(\Omega), g)$ follows 
\begin{equation*}
\partial_t\mathcal{U}(t,u)+\frac{1}{2}\int_\Omega |\frac{\delta}{\delta u(x)}\mathcal{U}(t,u)|^2u(x)dx+\mathcal{F}(u)=0.
\end{equation*}
Its ``characteristics'' in $(\mathcal{M}(\Omega), g)$ satisfy 
\begin{equation*}
\left\{\begin{aligned}
&\partial_tu-u\Phi =0,\\
&\partial_t\Phi+\frac{1}{2}|\Phi|^2+\frac{\delta}{\delta u}\mathcal{F}(u)=0.
\end{aligned}\right.
\end{equation*}
  \end{example}

\begin{example}
Consider 
\begin{equation*}
F(u)=u,\quad G(u)=u\log u-u, \quad R(u)=-u^\alpha\log u,
\end{equation*}
where $\alpha\in\mathbb{R}$ is a given value. In this case, 
  \begin{equation*}
    V_1(u)=\frac{F'(u)}{G''(u)}=u, \quad V_2(u)=-\frac{R(u)}{G'(u)}=u^\alpha. 
  \end{equation*}
The metric forms 
  \begin{equation*}
g(u)(\sigma_1,\sigma_2)=\int_\Omega (\nabla \Phi_1(x),\nabla \Phi_2(x))u(x)dx+\int_\Omega \Phi_1(x)\Phi_2(x)u(x)^\alpha dx,
  \end{equation*}
  with $\sigma_i=-\nabla\cdot(u\nabla\Phi_i)+\Phi_i u^\alpha$, $i=1,2$. The gradient flow of $\mathcal{G}$ in $(\mathcal{M}(\Omega), g)$ forms 
  \begin{equation*}
  \partial_t u=\Delta u+u^\alpha\log u, 
  \end{equation*}
  and the dissipation of $\mathcal{G}(u)$ satisfies  
  \begin{equation*}
  \mathcal{I}(u)=\int_\Omega \|\nabla \log u\|^2 udx+\int_\Omega |\log u|^2u^\alpha dx. 
  \end{equation*}
  And the Hamilton-Jacobi equation in $(\mathcal{M}(\Omega), g)$ follows 
\begin{equation*}
\partial_t\mathcal{U}(t,u)+\frac{1}{2}\int_\Omega \|\nabla \frac{\delta}{\delta u(x)}\mathcal{U}(t,u)\|^2u(x)dx+\frac{1}{2}\int_\Omega |\frac{\delta}{\delta u(x)}\mathcal{U}(t, u)|^2u(x)^\alpha dx+\mathcal{F}(u)=0.
\end{equation*}
Its ``characteristics'' in $(\mathcal{M}(\Omega), g)$ satisfy 
\begin{equation*}
\left\{\begin{aligned}
&\partial_tu+\nabla\cdot(u\nabla\Phi)=\Phi u^\alpha,\\
&\partial_t\Phi+\frac{1}{2}\|\nabla\Phi\|^2+\alpha \Phi u^{\alpha-1}+\frac{\delta}{\delta u}\mathcal{F}(u)=0.
\end{aligned}\right.
\end{equation*}
\end{example}
\begin{example}[Constant regularized optimal transport metric]
Consider 
\begin{equation*}
F(u)=u,\quad G(u)=(u+1)\log (u+1),\quad R(u)=0.
\end{equation*}
Thus 
\begin{equation*}
V_1(u)=\frac{F'(u)}{G''(u)}=u+1,\quad V_2(u)=0. 
\end{equation*}
The metric forms 
\begin{equation*}
g(u)(\sigma_1, \sigma_2)=\int_\Omega (\nabla\Phi_1, \nabla\Phi_2)(u+1)dx,
\end{equation*}
with $\sigma_i=-\nabla\cdot((u+1)\nabla\Phi_i)$, $i=1,2$. The gradient flow of $\mathcal{G}$ in $(\mathcal{M}(\Omega), g)$ satisfies 
\begin{equation*}
\partial_tu(t,x)=\Delta u(t,x). 
\end{equation*}
And the dissipation of $\mathcal{G}(u)$ satisfies 
\begin{equation*}
\mathcal{I}(u)=\int_\Omega\|\nabla \log(u+1)\|^2 (u+1)dx.  
\end{equation*}
And the Hamilton-Jacobi equation in $(\mathcal{M}(\Omega), g)$ follows 
\begin{equation*}
\partial_t\mathcal{U}(t,u)+\frac{1}{2}\int_\Omega \|\nabla\frac{\delta}{\delta u(x)}\mathcal{U}(t,u)\|^2(u+1)dx+\mathcal{F}(u)=0.
\end{equation*}
Its ``characteristics'' in $(\mathcal{M}(\Omega), g)$ satisfy
\begin{equation*}
\left\{\begin{aligned}
&\partial_tu+\nabla\cdot((u+1)\nabla\Phi)=0,\\
&\partial_t\Phi+\frac{1}{2}\|\nabla\Phi\|^2+\frac{\delta}{\delta u}\mathcal{F}(u)=0.
\end{aligned}\right.
\end{equation*}
\end{example}
\begin{example}[Fisher-KPP metric and Fisher-KPP equation]\label{FKPP}
Consider the Fisher-KPP equation
\begin{equation*}
\partial_tu=u(1-u)+\Delta u. 
\end{equation*}
Consider 
\begin{equation*}
F(u)=u,\quad G(u)=u\log u-u,\quad R(u)=u(1-u). 
\end{equation*}
Thus  
\begin{equation*}
V_1(u)=\frac{F'(u)}{G''(u)}=u,\quad V_2(u)=-\frac{R(u)}{G'(u)}=\frac{u(u-1)}{\log u}.  
\end{equation*}
The metric forms 
\begin{equation*}
g(u)(\sigma_1, \sigma_2)=\int_\Omega (\nabla\Phi_1, \nabla\Phi_2)udx+\int_\Omega \Phi_1\Phi_2 \frac{u(u-1)}{\log u} dx,
\end{equation*}
with $\sigma_i=-\nabla\cdot(u\nabla\Phi_i)+\frac{u(u-1)}{\log u}\Phi_i$, $i=1,2$. The gradient flow of $\mathcal{G}$ in $(\mathcal{M}(\Omega), g)$ satisfies the Fisher-KPP equation 
\begin{equation*}
\partial_tu(t,x)=u(1-u)+\Delta u. 
\end{equation*}
And the dissipation of $\mathcal{G}(u)$ satisfies 
\begin{equation*}
\mathcal{I}(u)=\int_\Omega\|\nabla \log u\|^2  u dx+\int_\Omega |\log u|^2\frac{u(u-1)}{\log u} dx.  
\end{equation*}
And the Hamilton-Jacobi equation in $(\mathcal{M}(\Omega), g)$ follows 
\begin{equation*}
\partial_t\mathcal{U}(t,u)+\frac{1}{2}\int_\Omega \|\nabla\frac{\delta}{\delta u(x)}\mathcal{U}(t,u)\|^2 udx+\frac{1}{2}\int_\Omega |\frac{\delta}{\delta u(x)}\mathcal{U}(t,u)|^2\frac{u(u-1)}{\log u}dx+\mathcal{F}(u)=0.
\end{equation*}
Its ``characteristics'' in $(\mathcal{M}(\Omega), g)$ satisfy 
\begin{equation*}
\left\{\begin{aligned}
&\partial_tu+\nabla\cdot(u\nabla\Phi)-\frac{u(u-1)}{\log u}\Phi=0,\\
&\partial_t\Phi+\frac{1}{2}\|\nabla\Phi\|^2+|\Phi|^2\frac{(2u-1)\log u+1-u}{(\log u)^2}+\frac{\delta}{\delta u}\mathcal{F}(u)=0.
\end{aligned}\right.
\end{equation*}

\end{example}

\begin{example}[Allen-Cahn metric and Allen-Cahn equation]
Let $f\in C^2(\mathbb{R})$ be a given function. Consider 
\begin{equation*}
F(u)=u,\quad G(u)=f(u),\quad R(u)=-f'(u). 
\end{equation*}
Thus  
\begin{equation*}
V_1(u)=\frac{F'(u)}{G''(u)}=f''(u)^{-1},\quad V_2(u)=-\frac{R(u)}{G'(u)}=1.
\end{equation*}
The metric forms 
\begin{equation*}
g(u)(\sigma_1, \sigma_2)=\int_\Omega (\nabla\Phi_1, \nabla\Phi_2)f''(u)^{-1}dx+\int_\Omega \Phi_1\Phi_2 dx,
\end{equation*}
with $\sigma_i=-\nabla\cdot(f''(u)^{-1}\nabla\Phi_i)+\Phi_i$, $i=1,2$. The gradient flow of $\mathcal{G}$ in $(\mathcal{M}(\Omega), g)$ satisfies 
\begin{equation*}
\partial_tu(t,x)=\Delta u(t,x)-f'(u(t,x)). 
\end{equation*}
And the dissipation of $\mathcal{G}(u)$ satisfies 
\begin{equation*}
\mathcal{I}(u)=\int_\Omega\|\nabla f'(u)\|^2 f''(u)^{-1}dx+\int_\Omega |f'(u)|^2 dx.  
\end{equation*}
And the Hamilton-Jacobi equation in $(\mathcal{M}(\Omega), g)$ follows 
\begin{equation*}
\partial_t\mathcal{U}(t,u)+\frac{1}{2}\int_\Omega \|\nabla\frac{\delta}{\delta u(x)}\mathcal{U}(t,u)\|^2f''(u(x))^{-1}dx+\frac{1}{2}\int_\Omega |\frac{\delta}{\delta u(x)}\mathcal{U}(t,u)|^2dx+\mathcal{F}(u)=0.
\end{equation*}
Its ``characteristics'' in $(\mathcal{M}(\Omega), g)$ satisfy 
\begin{equation*}
\left\{\begin{aligned}
&\partial_tu+\nabla\cdot(f''(u)^{-1}\nabla\Phi)-\Phi=0,\\
&\partial_t\Phi-\frac{1}{2}\|\nabla\Phi\|^2\frac{f'''(u)}{f''(u)^{2}}+\frac{\delta}{\delta u}\mathcal{F}(u)=0.
\end{aligned}\right.
\end{equation*}
\end{example}

\section{Algorithms}\label{sec5}
In this section, we propose an algorithm to solve the mean-field information variational problem~\eqref{variation} in two dimensions. Section~\ref{subsection:PDHG} presents the main optimization tool we use to solve the variational problem. We use the primal-dual hybrid gradient (PDHG) algorithm~\cite{champock11,champock16}, which is a popular first-order optimization method to solve saddle point problems. More specifically, we use the general proximal primal dual hybrid gradient (G-prox PDHG) method from~\cite{JacobsLegerLiOsher2018_solving}, which is a variation of the PDHG algorithm with a precondition matrix. Section~\ref{subsection:implementation} shows the implementation of G-Prox PDHG algorithm to solve the variational problem. Section~\ref{ssec:linear_case} provides additional algorithm when $V_1$ and $V_2$ in~\eqref{variation} are affine functions with specific forms (see~\eqref{eq:V_1_V_2}).
In Section~\ref{subsection:discretization}, we give details of the discretization of the algorithms to solve the variational problem on a compact set in 2-dimensional space. This section also shows the solution of the algorithm is equivalent to the solution of an implicit finite difference scheme that is stable and convergent for all ratios of $\Delta t$ and $\Delta x$.

\subsection{PDHG for mean-field control problems}\label{subsection:PDHG}
We first review the PDHG algorithm. Consider the following convex optimization problem.
\[
    \begin{aligned}
        \min_z\, f(Az) + g(z),
    \end{aligned}
\]
where $z$ is a variable to be minimized, $f$ and $g$ are convex functions and $A$ is a linear operator. Recall the Legendre transform $f^*$ of $f$ is
\[
    f^*(p) = \sup_{z} \, \langle z, p \rangle - f(z).
\]
It is well-known that if $f$ is convex then $f^{**} = f$. Thus, we have
\[
    f(w) = f^{**}(w) = \sup_p \langle w, p \rangle - f^*(p). 
\]
Using this property of convex functions, the minimization problem can be converted to a saddle point problem
\begin{equation}\label{eq:saddle_point_probalem}
    \begin{aligned}
        \min_z\, \max_p\,  g(z) + \langle Az, p\rangle - f^*(p) =: \mathcal{L}(z,p)
    \end{aligned}
\end{equation}
where $\mathcal{L}$ is a Lagrangian functional. The PDHG algorithm solves the problem by iterating
\begin{equation}\label{eq:alg-PDHG}
    \begin{aligned}
        p^{(k+1)} &= \argmax_p \, \mathcal{L}(z^{(k+1)},p) - \frac{1}{2\sigma} \|p-p^{(k)}\|^2_{L^2}\\
        z^{(k+1)} &= \argmin_z \, \mathcal{L}(z,2 p^{(k+1)} - p^{(k)}) + \frac{1}{2\tau} \|z-z^{(k)}\|^2_{L^2}.
    \end{aligned}
\end{equation}
The scheme converges if the step sizes $\tau$ and $\sigma$ satisfy
\begin{equation} \label{eq:step-sizes-condition}
    \tau \sigma \|A^T A\|_{L^2} < 1,
\end{equation}
where $\|\cdot\|_{L^2}$ is the operator norm in $L^2$. G-Prox PDHG provides an appropriate choice of norms for the algorithm and the authors prove that choosing the proper norms allows the algorithm to have larger step sizes and faster convergence than the original PDHG algorithm. The G-prox PDHG iterates
\begin{equation}\label{eq:alg-G-Prox-PDHG}
    \begin{aligned}
        p^{(k+1)} &= \argmax_p \, \mathcal{L}(z^{(k+1)},p) - \frac{1}{2\sigma} \|p-p^{(k)}\|^2_{H},\\
        z^{(k+1)} &= \argmin_z \, \mathcal{L}(z,2 p^{(k+1)} - p^{(k)}) + \frac{1}{2\tau} \|z-z^{(k)}\|^2_{L^2}.
    \end{aligned}
\end{equation}
Note that the norm in the first line is changed to $H$ from $L^2$. The norm $\|\cdot\|_{H}$ is defined as
\[
    \|p\|^2_{H} = \|A^\top p\|^2_{L^2}.
\]
For example, in our problem, we define $z=(m_1,m_2,u)$ as a vector of functions, $p=\Phi$, and the linear operator $A$ as
\begin{equation}\label{eq:linear_op_A}
    A(m_1, m_2, u)(t,x) = \partial_t u(t,x) + \nabla \cdot m_1(t,x) - m_2(t,x).
\end{equation} 
Note that the differential operator $\partial_t$ and the divergence operator $\nabla \cdot$ are linear operators, which make $A$ a linear operator. 
We define the inner product as
\[
    \langle z_1, z_2 \rangle = \int^1_0 \int_\Omega (m_1)_1(t,x) \cdot (m_1)_2(t,x) + (m_2)_1(t,x) (m_2)_2(t,x) + u_1(t,x) u_2(t,x)  \,dx\,dt
\]
where $z_i = \big( (m_1)_i, (m_2)_i, u_i \big)$ for $i=1,2$,
and
\[
    \langle p_1, p_2 \rangle = \int^1_0 \int_\Omega \Phi_1(t,x) \Phi_2(t,x) \,dx\,dt.
\]
where $p_i = \Phi_i$ for $i=1,2$.
Thus, we have
\[
    \langle Az,p\rangle = \int^1_0 \int_\Omega \Phi(t,x)\,\Big(\partial_t u(t,x) +  \nabla \cdot m_1(t,x) - m_2(t,x)\Big)\,dx\,dt.
\]
These inner products induce the $L^2$ norm of $z$ and $p$, such that 
\begin{equation}\label{eq:l2_norm_of_z_p}
\begin{aligned}
    \|z\|^2_{L^2} &= \int^1_0 \int_\Omega |m_1(t,x) |^2 + m_2(t,x)^2 + u(t,x)^2\,dx\,dt\\
    \|p\|^2_{L^2} &= \int^1_0 \int_\Omega \Phi^2(t,x)\,dx\,dt
\end{aligned}
\end{equation}
and the $H$ norm of $p$ can be written as
\begin{equation}\label{eq:linear_op_A_explicit}
    \|p\|^2_H = \int^1_0 \int_\Omega (\partial_t \Phi(t,x))^2 + | \nabla \Phi(t,x)|^2 + ( \Phi(t,x))^2\,dx\,dt
\end{equation}
where the $A^\top$ is computed using integration by parts. With abuse of notation, we also define the $L^2$ norm of $m_1$, $m_2$, and $u$ as follows
\begin{equation}\label{eq:l2_norm_of_u_m1_m2}
    \begin{aligned}
        \|m_1\|^2_{L^2} &= \int^1_0 \int_\Omega |m_1(t,x)|^2 \, dx\,dt\\
        \|m_2\|^2_{L^2} &= \int^1_0 \int_\Omega m_2(t,x)^2 \, dx\,dt\\
        \|u\|^2_{L^2} &= \int^1_0 \int_\Omega u(t,x)^2 \, dx\,dt\\
    \end{aligned}
\end{equation}
By choosing such norm based on~\cite{JacobsLegerLiOsher2018_solving}, the step sizes of the algorithm only need to satisfy
\[
    \sigma \tau < 1,
\]
which is independent of the operator $A$. From the definition of $A$ in \eqref{eq:linear_op_A}, the operator involves an unbounded operator $\nabla$. Thus, this step size condition in PDHG allows us to run the algorithm with larger step sizes independent of the grid sizes.

\subsection{Implementations of the algorithm}\label{subsection:implementation}

To implement the algorithm in the variational problem \eqref{variation}, we define $z$, $p$, and the linear operator $A$ as above.
Furthermore, we set convex functionals $g$ as
\begin{equation}
    \begin{split}
        g(m_1,m_2,u) &= \int^1_0 \int_\Omega \frac{|m_1(t,x)|^2}{2V_1(u(t,x))}+\frac{m_2(t,x)^2}{2V_2(u(t,x))}dx - \mathcal{F}(u(t,\cdot)) \, dt + \mathcal{G}(u(1,\cdot))
    \end{split}
\end{equation}
where the functional $\mathcal{F}$ is of the form
\[
    \mathcal{F}(u(t,\cdot)) = - \int_\Omega s(u(t,x)) \,dx
\]
and the terminal functional $\mathcal{G}$ is of the form
\[
    \mathcal{G}(u(1,\cdot)) = \int_\Omega G(u(1,x)) \,dx
\]
where $s,G:\mathbb{R} \rightarrow \mathbb{R}$ are convex functions. Set a convex functional $f$ as
\begin{equation}
    \begin{split}
        f\left(w\right) &= \int^1_0 \int_\Omega i_\infty \big(w(t,x)\big)\,dx\,dt\\
    i_\infty(t) &= \begin{cases}
     0 & \text{if } t = 0\\
     \infty & \text{otherwise}.
    \end{cases}
    \end{split}
\end{equation}
From the definition of function $f$, we can compute the Legendre transform of $f$
\begin{equation*}
    f^*(\Phi) = \sup_{w} \int^1_0 \int_\Omega \Phi(t,x)\, w(t,x) - i_\infty(w(t,x)) \,dx\,dt = 0.
\end{equation*}
Thus, from the definition of a Lagrangian functional in~\eqref{eq:experiment_lagrangian}, we have
\begin{equation}\label{eq:experiment_lagrangian}
    \begin{aligned}
        &\mathcal{L}(m_1, m_2,u,\Phi) \\
        &= g(m_1,m_2,u) + \langle A(m_1,m_2,u), \Phi \rangle - f^*(\Phi)\\    
        &= \int_0^1 \int_\Omega \Big\{ \frac{|m_1(t,x)|^2}{2V_1(u(t,x))}+\frac{m_2(t,x)^2}{2V_2(u(t,x))}\\
&\hspace{1cm}+\Phi(t,x)\Big(\partial_t u(t,x)+ \nabla\cdot m_1(t,x)-m_2(t,x)\Big)\Big\} dx  - \mathcal{F}(u(t,\cdot)) \,dt + \mathcal{G}(u(1,\cdot)).
    \end{aligned}
\end{equation}
G-Prox PDHG computes the saddle point $(m_1^*, m_2^*, u^*, \Phi^*)$ by iterating
\begin{equation}\label{eq:algorithm}
\begin{aligned}
    \Phi^{(k+ 1)}&= \argmax_{\Phi:[0,1]\times\Omega\rightarrow\mathbb{R}} \mathcal{L}(m_1^{(k)}, m_2^{(k)}, u^{(k)}, \Phi) - \frac{1}{2\sigma} \|\Phi - \Phi^{(k)}\|^2_{H}\\
    m_1^{(k+1)}&= \argmin_{m_1:[0,1]\times\Omega\rightarrow\mathbb{R}} \mathcal{L}(m_1, m_2^{(k)},u^{(k)},   2 \Phi^{(k+1)} - \Phi^{(k)}) + \frac{1}{2\tau} \|m_1 - m_1^{(k)}\|^2_{L^2}\\
    m_2^{(k+1)}&= \argmin_{m_2:[0,1]\times\Omega\rightarrow\mathbb{R}^d} \mathcal{L}(m_1^{(k)}, m_2,u^{(k)},   2 \Phi^{(k+1)} - \Phi^{(k)}) + \frac{1}{2\tau} \|m_2 - m_2^{(k)}\|^2_{L^2}\\
    u^{(k+1)} &= \argmin_{u:[0,1]\times\Omega\rightarrow\mathbb{R}} \mathcal{L}(m_1^{(k+1)}, m_2^{(k+1)}, u,   2 \Phi^{(k+1)} - \Phi^{(k)}) + \frac{1}{2\tau} \|u - u^{(k)}\|^2_{L^2}.
\end{aligned}
\end{equation}
Here, $L^2$ and $H$ norms are defined in~\eqref{eq:linear_op_A_explicit} and~\eqref{eq:l2_norm_of_u_m1_m2}.

From the optimality conditions, we can find the explicit formula for each variable $\Phi^{(k+1)},m_1^{(k+1)},m_2^{(k+1)},u^{(k+1)}(t,\cdot)$.

\begin{proposition}\label{prop:explicit-formula}
The variables $m_1^{(k+1)},m_2^{(k+1)},\Phi^{(k+1)},u^{(k+1)}(1,\cdot)$ from~\eqref{eq:algorithm} satisfy the following explicit formulas:
\begin{align*}
    \Phi^{(k+1)}(t,x) &= \Phi^{(k)}(t,x) + \sigma (A A^T)^{-1} \Bigl( \partial_t u^{(k)}(t,x) + \nabla \cdot m_1^{(k)}(t,x) - m_2^{(k)}(t,x) \Bigr),\\
   m_1^{(k+1)}(t,x) &= \frac{V_1(u^{(k)}(t,x))}{\tau + V_1(u^{(k)}(t,x))} \left( m_1^{(k)}(t,x) + \tau \nabla \big(2\Phi^{(k+1)}(t,x) - \Phi^{(k)}(t,x) \big) \right),\\
   m_2^{(k+1)}(t,x) &= \frac{V_2(u^{(k)}(t,x))}{\tau + V_2(u^{(k)}(t,x))} \left( m_2^{(k)}(t,x) + \tau \big(2\Phi^{(k+1)}(t,x) - \Phi^{(k)}(t,x)\big) \right),
\end{align*}
for $(t,x)\in[0,1]\times\Omega$ and
\[
    u^{(k+1)}(1,x) = (G^*)'(-2\Phi^{(k+1)}(1,x)+\Phi^{(k)}(1,x)),\quad x\in\Omega.
\]
\end{proposition}
The variable $u^{(k+1)}$ also satisfies the following optimality condition:
\begin{equation}\label{eq:newton_method_equation}
    \begin{aligned}
     - \frac{\|m_1^{(k+1)}\|^2}{2V_1^2(u^{(k+1)})} V_1'(u^{(k+1)}) &- \frac{|m_2^{(k+1)}|^2}{2V_2^2(u^{(k+1)})} V_2'(u^{(k+1)})\\
     &- \partial_t (2\Phi^{(k+1)} - \Phi^{(k)}) - s' (u^{(k+1)}) + \frac{1}{\tau} (u^{(k+1)}-u^{(k)})= 0
    \end{aligned}
\end{equation}
for $(t,x) \in [0,1]\times \Omega$. As one can see, the solution $u^{(k+1)}$ of equation depends on $V_1$, $V_2$, and $\mathcal{F}$. We first present the algorithm for general $V_1$, $V_2$ and $\mathcal{F}$.  Assuming $V_1$, $V_2$ and $\mathcal{F}$ are smooth, we can compute $u^{(k+1)}$ using Newton's method.
\begin{tabbing}\label{alg:newton-u}
aaaaa\= aaa \=aaa\=aaa\=aaa\=aaa=aaa\kill  
   \rule{\linewidth}{0.8pt}\\
   \noindent{\large\bf Algorithm 1: Newton's method to compute $u^{(k+1)}(t,x)$}\\
   \textbf{Input}: $u^{(k)}(t,x), m_1(t,x), m_2(t,x), \partial_t \Phi(t,x)$.\\
   \textbf{Output}: $u^{(k+1)}(t,x)$.\\
   \rule{\linewidth}{0.5pt}\\
   Initialize $u^0 = u^{(k)}(t,x)$ and set $\epsilon>0$.\\
   \+\+\textbf{For} $j\in\mathbb{N}$\\
   [2ex]   
   \' $a$ \'  $= - \frac{\|m_1\|^2}{2V_1^2(u^{j})} V_1'(u^{j}) - \frac{|m_2|^2}{2V_2^2(u^{j})} V_2'(u^{j}) - \partial_t \Phi - s' (u^{j}) + \frac{1}{\tau} (u^{j}-u^{(k)});$\\
    [2ex]   
  \' $b$ \'  $= \frac{\|m_1\|^2}{V_1^3(u^{j})} (V_1'(u^{j}))^2 - \frac{\|m_1\|^2}{2 V_1^2(u^{j})} V_1''(u^{j}) + \frac{|m_2|^2}{V_2^3(u^{j})} (V_2'(u^{j}))^2 - \frac{|m_2|^2}{2 V_2^2(u^{j})} V_2''(u^{j})  + \frac{c}{u^{j}} + \frac{1}{\tau};$\\
    [2ex] 
    \-\- \' $u^{j+1}$ \'  $= \max(0, u^{j} - \frac{a}{b});$\\
    [2ex] 
    \hspace{0.6cm} Stop the iteration when $\left|\frac{a}{b}\right| < \epsilon$.\\
    [2ex] 
   \rule{\linewidth}{0.5pt}
\end{tabbing}
In the algorithm~1, we omit $(t,x)$ in $u^j$, $u^{j+1}$, $u^{(k)}$, $m_1$, $m_2$, $\partial_t \Phi$ to simplify the notations.
Using Proposition~\ref{prop:explicit-formula}, we can rewrite the algorithm~\eqref{eq:algorithm}. 
\newpage

\begin{tabbing}\label{alg:gproxPDHG}
aaaaa\= aaa \=aaa\=aaa\=aaa\=aaa=aaa\kill  
   \rule{\linewidth}{0.8pt}\\
   \noindent{\large\bf Algorithm 2: G-prox PDHG for mean-field information variational problem}\\
   \textbf{Input}: Initial density $u_0$.\\
   \textbf{Output}: $u, \Phi, m_2:[0,T]\times\Omega \rightarrow \mathbb{R}$, $m_1:[0,T]\times\Omega \rightarrow \mathbb{R}^d$.\\
   \rule{\linewidth}{0.5pt}\\
   \+\+\textbf{For} $k\in\mathbb{N}$\\
   [2ex]   
   \+\+\+ \hspace{-1.5cm} \textbf{For} $(t,x)\in[0,1]\times\Omega$\\
   [2ex]   
      \' $\Phi^{(k+ 1)}(t,x)$ \' $= \Phi^{(k)}(t,x) + \sigma (A A^T)^{-1} \Bigl( \partial_t u^{(k)}(t,x) + \nabla \cdot m_1^{(k)}(t,x) - m_2^{(k)}(t,x) \Bigr);$\\
    [2ex]   
      \' $m_1^{(k+1)}(t,x)$ \' $= \frac{V_1(u^{(k)}(t,x))}{\tau + V_1(u^{(k)}(t,x))} \left( m_1^{(k)} (t,x) + \tau \nabla \big(2\Phi^{(k+1)}(t,x) - \Phi^{(k)}(t,x) \big) \right);$\\
    [2ex]   
    \' $m_2^{(k+1)}(t,x)$ \' $= \frac{V_2(u^{(k)}(t,x))}{\tau + V_2(u^{(k)}(t,x))} \left( m_2^{(k)}(t,x) + \tau \big(2\Phi^{(k+1)}(t,x) -\Phi^{(k)} (t,x)\big) \right);$\\
   [2ex]   
    \' $u^{(k+1)}(t,x)$ \' $\leftarrow$ Compute using Algorithm~1  \\
     \hspace{2cm}  with $m_1=m_1^{(k+1)}, m_2=m_2^{(k+1)}, \Phi = 2\Phi^{(k+1)} - \Phi^{(k)};$ \\
    [2ex]
   \-\-\-\-\- \' $u^{(k+1)}(1,x)$ \' $ = (G^*)'(-2\Phi^{(k+1)}(1,x)+\Phi^{(k)}(1,x))$.\\
  [1ex]
   \rule{\linewidth}{0.5pt}
\end{tabbing}

The convergence of Newton's method from Algorithm~1 depends on functions $V_1$, $V_2$ and a functional $\mathcal{F}$. In the numerical experiments, we use 
\begin{equation*}
 \mathcal{F}(u(t,\cdot))(x) =  -c \int_\Omega \Big(u(t,x) \log u(t,x) - u(t,x)\Big)\, dx,  
\end{equation*}
 with a given positive constant $c$. This functional regularizes equation~\eqref{eq:newton_method_equation}. Thus the algorithm converges faster. In the actual numerical experiments, Newton's method only requires less than $10$ iterations to reach the residual error $|a/b|$ less than $10^{-10}$.  We use FFTW library to compute $(A A^T)^{-1}$ by Fast Fourier Transform (FFT). It takes $O(n\log n)$ operations per iteration, where $n$ is the number of points in both time and spatial grids. Overall, the algorithm takes just $O(n\log n)$ operations per iteration.

In the next subsection, we present an alternative algorithm when $V_1$ and $V_2$ are linear functions with specific forms. The optimality condition for $u^{(k+1)}$ leads to a third-order polynomial equation which has an analytical solution and can easily be computed without using the second-order optimization method just as in Algorithm~1. 

\subsection{Affine Case}\label{ssec:linear_case}

Suppose $V_1$ and $V_2$ are affine functions, such that
\begin{equation}\label{eq:V_1_V_2}
    V_1(u) = c_1 (u + c_3),\quad V_2(u) = c_2 (u + c_3)
\end{equation}
where $c_1,c_2,c_3$ are constants, and the functional $\mathcal{F}$ is of the form
\[
    \mathcal{F}(u(t,\cdot)) = - \int_\Omega s(u(t,x)) \,dx
\]
where $s:\mathbb{R} \rightarrow \mathbb{R}$ is a smooth convex function. Note that from the algorithm~\eqref{eq:algorithm}, the optimality condition for $u^{(k+1)}$ involves a functional $\mathcal{F}$. We can get rid of a functional $\mathcal{F}$ from the optimality condition  for $u^{(k+1)}$ by introducing an extra Lagrangian multiplier $\Psi:[0,1]\times \Omega \rightarrow \mathbb{R}$ to the saddle point problem~\eqref{eq:saddle_point_probalem}.
\begin{equation*}
    \begin{aligned}
        &\inf_{m_1,m_2,u} \sup_{\Phi,\Psi} \mathcal{L}(m_1, m_2,u,\Phi,\Psi)\\
        = &\inf_{m_1,m_2,u} \sup_{\Phi,\Psi} \int_0^1 \int_\Omega \Big\{ \frac{|m_1(t,x)|^2}{2V_1(u(t,x))}+\frac{m_2(t,x)^2}{2V_2(u(t,x))}\\
    &\hspace{2cm} + \Phi(t,x)\Big(\partial_t u(t,x)+ \nabla\cdot m_1(t,x)-m_2(t,x)\Big)\Big\} dx\,dt\\
    &\hspace{1.5cm} + \int^1_0 \int_\Omega \Psi(t,x) u(t,x)  - s^*\big(\Psi(t,x)\big)\,dx\,dt + \mathcal{G}(u(1,\cdot)).
    \end{aligned}
\end{equation*}
Using G-Prox PDHG algorithm, the saddle point can be computed by iterating
\begin{equation}\label{eq:algorithm_linear}
\begin{aligned}
    \Phi^{(k+ 1)}&= \argmax_\Phi \mathcal{L}(m_1^{(k)}, m_2^{(k)}, u^{(k)}, \Phi, \Psi^{(k)}) - \frac{1}{2\sigma} \|\Phi - \Phi^{(k)}\|^2_{H}\\
    \Psi^{(k+ 1)}&= \argmax_\Psi \mathcal{L}(m_1^{(k)}, m_2^{(k)}, u^{(k)}, \Phi^{(k)}, \Psi) - \frac{1}{2\sigma} \|\Psi - \Psi^{(k)}\|^2_{L^2}\\
    m_1^{(k+1)}&= \argmin_m \mathcal{L}(m, m_2^{(k)},u^{(k)},   2 \Phi^{(k+1)} - \Phi^{(k)},  2 \Psi^{(k+1)} - \Psi^{(k)}) + \frac{1}{2\tau} \|m - m_1^{(k)}\|^2_{L^2}\\
    m_2^{(k+1)}&= \argmin_m \mathcal{L}(m_1^{(k)}, m,u^{(k)},   2 \Phi^{(k+1)} - \Phi^{(k)},  2 \Psi^{(k+1)} - \Psi^{(k)}) + \frac{1}{2\tau} \|m - m_2^{(k)}\|^2_{L^2}\\
    u^{(k+1)} &= \argmin_u \mathcal{L}(m_1^{(k+1)}, m_2^{(k+1)}, u,   2 \Phi^{(k+1)} - \Phi^{(k)},  2 \Psi^{(k+1)} - \Psi^{(k)}) + \frac{1}{2\tau} \|u - u^{(k)}\|^2_{L^2}.
\end{aligned}
\end{equation}

From the optimality conditions, we can find the explicit formula for each variable $\Phi^{(k+1)}, \Psi^{(k+1)},m_1^{(k+1)},m_2^{(k+1)}$.

\begin{proposition}\label{prop:explicit-formula_linear}
The variables $m_1^{(k+1)},m_2^{(k+1)},\Phi^{(k+1)},\Psi^{(k+1)},u^{(k+1)}(1,\cdot)$ from~\eqref{eq:algorithm_linear} satisfy the following explicit formulas:
\begin{align*}
    \Phi^{(k+1)}(t,x) &= \Phi^{(k)}(t,x) + \sigma (A A^T)^{-1} \Bigl( \partial_t u^{(k+1)}(t,x) + \nabla \cdot m_1^{(k+1)}(t,x) - m_2^{(k+1)}(t,x) \Bigr),\\
    \Psi^{(k+1)}(t,x) &= \big({\textrm{Id}} + \sigma (s^*)'\big)^{-1} \big(\Psi^{(k)}(t,x) + \sigma u^{(k)}(t,x) \big),\\
   m_1^{(k+1)}(t,x) &= \frac{V_1(u^{(k)}(t,x))}{\tau + V_1(u^{(k)}(t,x))} \left( m_1^{(k)}(t,x) + \tau \nabla \big(2\Phi^{(k+1)}(t,x) - \Phi^{(k)}(t,x)\big) \right),\\
   m_2^{(k+1)}(t,x) &= \frac{V_2(u^{(k)}(t,x))}{\tau + V_2(u^{(k)}(t,x))} \left( m_2^{(k)}(t,x) + \tau \big(2\Phi^{(k+1)}(t,x) - \Phi^{(k)}(t,x)\big) \right),
\end{align*}
for $(t,x)\in[0,1]\times\Omega$,
\[
    u^{(k+1)}(1,x) = (G^*)'(- 2\Phi^{(k+1)}(1,x) + \Phi^{(k)}(1,x)),\quad x\in\Omega,
\]
and ${\textrm{Id}}$ is an identity operator.
\end{proposition}
For the explicit formula of $\Psi^{(k+1)}$, if $s(t) = \frac{c}{2} t^2$ with $c>0$, then the formula can be simplified to
\[
    \Psi^{(k+1)}(t,x) = \frac{1}{1+\sigma/c} \big( \Psi^{(k)}(t,x) + \sigma u^{(k)}(t,x) \big).
\]
The following proposition shows the equation from the optimality condition of $u^{(k+1)}$ from~\eqref{eq:algorithm}.

\begin{proposition}\label{prop:optimality_u_linear}
    If $V_1$ and $V_2$ are affine functions in~\eqref{eq:V_1_V_2},
    $u^{(k+1)}(t,x)$ satisfies the following equation for $(t,x)\in[0,1]\times\Omega$:
    \begin{equation}\label{eq:optimality_of_u_linear}
        \big(u^{(k+1)}(t,x)\big)^3 + \big(u^{(k+1)}(t,x) 
        \big)^2 ( 2c_3 + k_2) + u^{(k+1)}(t,x) (c_3^2 + 2c_3 k_2) + (k_1 + k_2 c_3^2) = 0
    \end{equation}
    where
    \[
    \begin{aligned}
        k_1 (t,x) &= - \frac{\tau |m_1^{(k+1)}(t,x)|^2}{2c_1} - \frac{\tau m_2^{(k+1)}(t,x)^2}{2c_2},
        \\
        k_2 (t,x) &= - \tau \partial_t \big(2 \Phi^{(k+1)}(t,x) - \Phi^{(k)}(t,x) \big) + \tau \big(2 \Psi^{(k+1)}(t,x) - \Psi^{(k)}(t,x) \big) - u^{(k)}(t,x).
    \end{aligned}
    \]
\end{proposition}
\begin{proof}
    The optimality condition for $u^{(k+1)}$ from~\eqref{eq:algorithm} gives
    \begin{equation*}
        - \frac{|m_1^{(k+1)}|^2}{2c_1 (u^{(k+1)}+c_3)^2} - \frac{(m_2^{(k+1)})^2}{2c_2 (u^{(k+1)}+c_3)^2}
        - \partial_t (2\Phi^{(k+1)} - \Phi^{(k)})
        + 2\Psi^{(k+1)} - \Psi^{(k)} + \frac{1}{\tau} (u^{(k+1)} - u^{(k)}) = 0.
    \end{equation*}
    Simple algebra leads to~\eqref{eq:optimality_of_u_linear}. 
    
\end{proof}
From Proposition~\ref{prop:optimality_u_linear}, the optimality condition for $u^{(k+1)}$ leads to the third order polynomial which has an analytical solution. Thus, $u^{(k+1)}$ can be computed through
\[
    u^{(k+1)}(t,x) = root_+(2c_3 +k_2, c_3^2 + 2c_3 k_2, k_1+k_2 c_3^2),\quad (t,x)\in[0,1] \times \Omega
\]
where $root_+(q_1, q_2, q_3)$ is a positive root of a third-order polynomial $x^3 + q_1 x^2 + q_2 x + q_3 = 0$. Together with Proposition~\ref{prop:explicit-formula_linear} and~\ref{prop:optimality_u_linear}, we present the algorithm for the linear case.
\begin{tabbing}\label{alg:gproxPDHG_linear}
aaaaa\= aaa \=aaa\=aaa\=aaa\=aaa=aaa\kill  
   \rule{\linewidth}{0.8pt}\\
   \noindent{\large\bf Algorithm~3: G-prox PDHG for mean-field information variational problem}\\
   \noindent{\hspace{2.7cm} \large\bf with $V_1$, $V_2$ defined in~\eqref{eq:V_1_V_2}}\\
   \textbf{Input}: Initial density $u_0$.\\
   \textbf{Output}: $u, \Phi, m_2:[0,T]\times\Omega \rightarrow \mathbb{R}$, $m_1:[0,T]\times\Omega \rightarrow \mathbb{R}^d$.\\
   \rule{\linewidth}{0.5pt}\\
   \+\+\textbf{For} $k\in\mathbb{N}$\\
   [2ex]   
   \+\+ \hspace{-1.5cm} \textbf{For} $(t,x)\in[0,1]\times\Omega$\\
   [2ex]   
      \' $\Phi^{(k+ 1)}(t,x)$ \' $= \Phi^{(k)}(t,x) + \sigma (A A^T)^{-1} \Bigl( \partial_t u^{(k)}(t,x) + \nabla \cdot m_1^{(k)}(t,x) - m_2^{(k)}(t,x) \Bigr);$\\
   [2ex]   
      \' $\Psi^{(k+1)}(t,x)$ \' $= \big({\textrm{Id}} + \sigma (s^*)'\big)^{-1} \big(\Psi^{(k)}(t,x) + \sigma u^{(k)}(t,x) \big);$\\
    [2ex]   
      \' $m_1^{(k+1)}(t,x)$ \' $= \frac{V_1(u^{(k)} (t,x))}{\tau + V_1(u^{(k)} (t,x))} \left( m_1^{(k)} (t,x) + \tau \nabla \big(2\Phi^{(k+1)} (t,x) - \Phi^{(k)} (t,x)\big) \right);$\\
    [2ex]   
       \' $m_2^{(k+1)}(t,x)$ \' $= \frac{V_2(u^{(k)} (t,x))}{\tau + V_2(u^{(k)} (t,x))} \left( m_2^{(k)} (t,x) + \tau \big(2\Phi^{(k+1)} (t,x) -\Phi^{(k)} (t,x)\big) \right);$\\
    [2ex]
   \' $u^{(k+1)}(t,x)$ \' $ = root_+(2b+k_2, b^2 + 2bk_2, k_1+k_2 b^2)$,\quad $k_1$, $k_2$ are defined in Proposition~\ref{prop:optimality_u_linear}\\
  [2ex]
  \-\-\-\- \' $u^{(k+1)}(1,x)$ \' $ = (G^*)'(-2\Phi^{(k+1)}(1,x)+\Phi^{(k)}(1,x))$.\\
  [1ex]
   \rule{\linewidth}{0.5pt}
\end{tabbing}
Note that by introducing an extra dual variable $\Psi$, the optimality condition for $u^{(k+1)}$ becomes a third-order polynomial which has an analytical solution. Thus, the linear case algorithm does not require a second-order optimization method to compute $u^{(k+1)}$ as in Algorithm~2. The computational complexity of Algorithm~3 is similar to Algorithm~2. 
We use FFTW library to compute $(A A^T)^{-1}$ by Fast Fourier Transform (FFT), which takes $O(n\log n)$ operations per iteration, where $n$ is the number of points in time and spatial grids. The other variables, $\Psi$, $m_1$, $m_2$, and $u$, takes $O(n)$ operations to compute over $(t,x)\in[0,1]\times\Omega$. Overall, the algorithm takes $O(n\log n)$ operations per iteration.

\subsection{Discretization in 2D}\label{subsection:discretization}
Let $\Omega = [0,1]^2$ be a unit square in $\mathbb{R}^2$ and the terminal time $T=1$. Since we simulate the formulation on a compact set, we set the following boundary conditions on $m_1$ for the variational problem~\eqref{variation}.
\begin{equation}\label{eq:discrete_boundary_conditions}
    \begin{aligned}
        m_1(t,x) \cdot \vec{n}(t,x) = 0, \quad& (t,x)\in [0,1]\times \partial \Omega
    \end{aligned}
\end{equation}
where $\vec{n}$ is an outward normal vector and $\partial \Omega$ is a boundary of $\Omega$. This boundary condition is equivalent to no-flux condition that means no mass flows through a boundary. For the numerical experiments,
the domain $[0,1]\times\Omega$ is discretized with the regular Cartesian grid. Denote
\[\Delta x_1 = \frac{1}{N_{x_1}},\quad \Delta x_2 = \frac{1}{N_{x_2}},\quad \Delta t = \frac{1}{N_t-1}\]
where $N_{x_1}$, $N_{x_2}$ are the number of discretized points in $x_1$-axis and $x_2$-axis, and $N_t$ is the number of discretized points in time. Denote grid points in space and time as
\begin{equation*}
    \begin{aligned}
        x_{jl} &= \left( (j+0.5) \Delta x_1, (l+0.5)\Delta x_2 \right), && j = 0,\cdots,N_{x_1}-1\\
         &  &&  l = 0,\cdots,N_{x_2}-1\\
        t_n &= n \Delta t , && n = 0,\cdots,N_t-1.
    \end{aligned}
\end{equation*}
Using notations, we have the following approximations:
\[
\begin{aligned}
    u(t_n, x_{jl}) \;\approx&\; u_{\overlinetmp{n,jl}},\\ 
    \Phi(t_n, x_{jl}) \;\approx&\; \Phi_{\overlinetmp{n,jl}},\\
     m_1(t_n, x_{jl}) \;\approx&\; [m_1]_{\overlinetmp{n,jl}} = ([m_{1,x_1}]_{\overlinetmp{n,jl}},[m_{1,x_2}]_{\overlinetmp{n,jl}}),\\
    m_2(t_n, x_{jl}) \;\approx&\; {[m_2]}_{\overlinetmp{n,jl}} .
\end{aligned}
\]
The subscript $n$ represents the approximation at $t_n\in[0,1]$ and the subscript $jl$ represents the approximation at $x_{jl} \in \Omega$. Here, $m_{1,x_1}$ and $m_{1,x_2}$ are $x_1$ and $x_2$ components of $m_1$, respectively.

 The gradient operators are calculated using the first order approximation.
\begin{align*}
    \partial_{x_1} \Phi_{\overlinetmp{n,jl}} &= \frac{\Phi_{\overlinetmp{n,jl}}-\Phi_{\overlinetmp{n,j-1,l}}}{\Delta x_1},
    & \partial_{x_2} \Phi_{\overlinetmp{n,jl}} &= \frac{\Phi_{\overlinetmp{n,jl}}-\Phi_{\overlinetmp{n,j,l-1}}}{\Delta x_2}\\
    \partial_t \Phi_{\overlinetmp{n,jl}} &= \frac{\Phi_{\overlinetmp{n+1,jl}} - \Phi_{\overlinetmp{n,jl}}}{\Delta t},
    & \partial_t u_{\overlinetmp{n,jl}} &= \frac{u_{\overlinetmp{n,jl}} - u_{\overlinetmp{n-1,jl}}}{\Delta t}
\end{align*}
where $\partial_{x_i}$ is a partial derivative with respect to $x_i$ ($i=1,2$) axis and $\partial_t$ is a partial derivative with respect to time $t$. The divergence operator for $m_1$ is calculated using the first order approximation as well.
\begin{align*}
    \nabla \cdot [m_1]_{\overlinetmp{n,jl}} = \partial_{x_1} [m_{1,x_1}]_{\overlinetmp{n,jl}} + \partial_{x_2} [m_{1,x_2}]_{\overlinetmp{n,jl}}
\end{align*}
\begin{align*}
    \partial_{x_1} [m_{1,x_1}]_{\overlinetmp{n,jl}} &= \frac{[m_{1,x_1}]_{\overlinetmp{n,j+1,l}}-[m_{1,x_1}]_{\overlinetmp{n,j,l}}}{\Delta x_1},
    & \partial_{x_2} [m_{1,x_2}]_{\overlinetmp{n,jl}} &= \frac{[m_{1,x_2}]_{\overlinetmp{n,j,l+1}}-[m_{1,x_2}]_{\overlinetmp{n,jl}}}{\Delta x_2}
\end{align*}
\noindent Note that the partial derivative of $\Phi$ with respect to $t$ uses forward difference scheme and that of $u$ uses backward difference scheme. The partial derivatives of $\Phi$ with respect to $x_1$ and $x_2$ use backward difference scheme and those of $m_1$ use forward difference scheme. The following proposition justifies our choices of finite difference schemes for each variable.

\begin{proposition}\label{prop:discrete_optimality}
    Define the discrete partial differential operators of $u$ and $m_1$ as follows
    \begin{equation}\label{eq:finite_difference_schemes_u_m_1}
    \begin{aligned}
         \partial_t u_{\overlinetmp{n,jl}} &= \frac{u_{\overlinetmp{n,jl}} - u_{\overlinetmp{n-1,jl}}}{\Delta t} \\
         \nabla \cdot [m_1]_{\overlinetmp{n,jl}} &= \partial_{x_1} [m_{1,x_1}]_{\overlinetmp{n,jl}} + \partial_{x_2} [m_{1,x_2}]_{\overlinetmp{n,jl}}\\
             \partial_{x_1} [m_{1,x_1}]_{\overlinetmp{n,jl}} &= \frac{[m_{1,x_1}]_{\overlinetmp{n,j+1l}}-[m_{1,x_1}]_{\overlinetmp{n,jl}}}{\Delta x_1}\\
        \partial_{x_2} [m_{1,x_2}]_{\overlinetmp{n,jl}} &= \frac{[m_{1,x_2}]_{\overlinetmp{n,j,l+1}}-[m_{1,x_2}]_{\overlinetmp{n,jl}}}{\Delta x_2}.
    \end{aligned}
    \end{equation}
    Suppose $u$ satisfies the Dirichlet boundary condition in time
    \[
    \begin{aligned}
        u_{0,jl} = u_0(x_{jl}),\quad &j=0,\cdots,N_{x_1}-1\\
        &l=0,\cdots,N_{x_2}-1
    \end{aligned}
    \]
    and $m_1$ satisfies the no-flux boundary condition in space
    \begin{equation*}
    \begin{aligned}
        \,[m_{1,x_1}]_{\overlinetmp{n,0,l}} = [m_{1,x_1}]_{\overlinetmp{n,N_{x_1},l}} = 0,\quad &n=0,\cdots,N_t-1,\\ &l=0,\cdots,N_{x_2}-1
    \end{aligned}
    \end{equation*}
    and 
    \[
    \begin{aligned}
        \,[m_{1,x_2}]_{\overlinetmp{n,j,0}} = [m_{1,x_2}]_{\overlinetmp{n,j, N_{x_2}}} = 0,\quad &n=0,\cdots,N_t-1,\\ &j=0,\cdots,N_{x_1}-1.
    \end{aligned}
    \]
    Then, $\Phi_{\overlinetmp{n,jl}}$, $[m_1]_{\overlinetmp{n,jl}}$, $[m_2]_{\overlinetmp{n,jl}}$ satisfy the following optimality conditions:
    \begin{align}\label{eq:optimality_phi_discrete}
        \partial_t u_{\overlinetmp{n,jl}} + \nabla \cdot [m_1]_{\overlinetmp{n,jl}} - [m_2]_{\overlinetmp{n,jl}} = 0
    \end{align}
    \begin{align}\label{eq:optimality_m1_discrete}
        \frac{[m_{1,x_i}]_{\overlinetmp{n,jl}}}{V_1(u_{\overlinetmp{n,jl}})} - \partial_{x_i} \Phi_{\overlinetmp{n,jl}} = 0, \quad (i=1,2)
    \end{align}
    \begin{align}\label{eq:optimality_m2_discrete}
        \frac{[m_{2}]_{\overlinetmp{n,jl}}}{V_2(u_{\overlinetmp{n,jl}})} - \Phi_{\overlinetmp{n,jl}} = 0,
    \end{align}
    for $n=0,\cdots,N_t-1$, $j=0,\cdots,N_{x_1}-1$, $l=0,\cdots,N_{x_2}-1$. The optimality condition for $u_{\overlinetmp{n,jl}}$ is
    \begin{equation}\label{eq:optimality_u_discrete}
    \begin{aligned}
        - \frac{([m_{1,x_1}]_{\overlinetmp{n,jl}})^2+([m_{1,x_2}]_{\overlinetmp{n,jl}})^2}{2 V_1(u_{\overlinetmp{n,jl}})^2} V_1'(u_{\overlinetmp{n,jl}})
            - \frac{([m_{2}]_{\overlinetmp{n,jl}})^2}{2 V_2(u_{\overlinetmp{n,jl}})^2} V_2'(u_{\overlinetmp{n,jl}})
            - \partial_{t} \Phi_{\overlinetmp{n,jl}} + s'(u_{\overlinetmp{n,jl}})= 0
    \end{aligned}
    \end{equation}
    for $n=1,\cdots,N_t-2$, $j=0,\cdots,N_{x_1}-1$, $l=0,\cdots,N_{x_2}-1$ and
    \begin{equation}\label{eq:terminal_phi_condition_discrete}
        \Phi_{\overlinetmp{N_t-1,jl}} + G'(u_{\overlinetmp{N_t-1,jl}}) = 0
    \end{equation}
    for $j=0,\cdots,N_{x_1}-1$, $l=0,\cdots,N_{x_2}-1$.
    The differential operators of $\Phi$ with respect to time and space are defined as
    \begin{align*}
        \partial_t \Phi_{\overlinetmp{n,jl}} &= \frac{\Phi_{\overlinetmp{n+1,jl}} - \Phi_{\overlinetmp{n,jl}}}{\Delta t},\\
        \partial_{x_1} \Phi_{\overlinetmp{n,jl}} &= \frac{\Phi_{\overlinetmp{n,jl}}-\Phi_{\overlinetmp{n,j-1,l}}}{\Delta x_1},\\
        \partial_{x_2} \Phi_{\overlinetmp{n,jl}} &= \frac{\Phi_{\overlinetmp{n,jl}}-\Phi_{\overlinetmp{n,j,l-1}}}{\Delta x_2},
    \end{align*}
with Neumann boundary conditions in time and space
    \[
    \begin{aligned}
        \Phi_{\overlinetmp{1,jl}} = \Phi_{\overlinetmp{0,jl}},\quad &j=0,\cdots,N_{x_1}-1,\\ &l=0,\cdots,N_{x_2}-1.
    \end{aligned}
    \]
    \[
    \begin{aligned}
        \Phi_{\overlinetmp{n,j,0}} = \Phi_{\overlinetmp{n,j,-1}},\quad &n=0,\cdots,N_t-1,\\ &j=0,\cdots,N_{x_1}-1.
    \end{aligned}
    \]
    \[
    \begin{aligned}
        \Phi_{\overlinetmp{n,0,l}} = \Phi_{\overlinetmp{n,-1,l}},\quad &n=0,\cdots,N_t-1,\\ &l=0,\cdots,N_{x_2}-1
    \end{aligned}
    \]
    and a boundary condition in time at $t=1$ from~\eqref{eq:terminal_phi_condition_discrete}.
    
\end{proposition}

\begin{proof}
    Using the discretization notations that are introduced in this subsection, we discretize the Lagrangian functional in~\eqref{eq:experiment_lagrangian}.
    \begin{equation*}
        \begin{aligned}
            &\mathcal{L}(m_1,m_2,u,\Phi)\\
            &\approx 
            \Delta t \Delta x_1 \Delta x_2 \sum^{N_{t}-1}_{n=1} \sum^{N_{x_1}-1}_{j=0} \sum^{N_{x_2}-1}_{l=0} \Bigg( \frac{([m_{1,x_1}]_{\overlinetmp{n,jl}})^2 + ([m_{1,x_1}]_{\overlinetmp{n,jl}})^2}{2 V_1(u_{\overlinetmp{n,jl}})}
            +  \frac{([m_{2}]_{\overlinetmp{n,jl}})^2}{2 V_2(u_{\overlinetmp{n,jl}})}\\
            &\hspace{5cm} + \Phi_{\overlinetmp{n,jl}} \Big( \partial_t u_{\overlinetmp{n,jl}} + \nabla \cdot [m_1]_{\overlinetmp{n,jl}} - [m_2]_{\overlinetmp{n,jl}} \Big) + s(u_{\overlinetmp{n,jl}})  \Bigg)\\
            &\quad+ \Delta x_1 \Delta x_2 \sum^{N_{x_1}-1}_{j=1} \sum^{N_{x_2}-1}_{l=1} G(u_{\overlinetmp{N_t-1,jl}}).
        \end{aligned}
    \end{equation*}
    The optimality condition for $\Phi_{\overlinetmp{n,jl}}$ can be obtained by differentiating with respect to $\Phi_{\overlinetmp{n,jl}}$
    \[
        \partial_t u_{\overlinetmp{n,jl}} + \nabla \cdot [m_1]_{\overlinetmp{n,jl}} - [m_2]_{\overlinetmp{n,jl}} = 0
    \]
    for $n=0,\cdots,N_t-1$, $j=0,\cdots,N_{x_1}-1$, $l=0,\cdots,N_{x_2}-1$.
    
    Let us take a closer look at the terms with gradients and divergence operators. Using the difference schemes defined in~\eqref{eq:finite_difference_schemes_u_m_1}, it can be rewritten as
    \begin{align*}
        & \sum^{N_{t}-1}_{n=1} \sum^{N_{x_1}-1}_{j=0} \sum^{N_{x_2}-1}_{l=0}\Phi_{\overlinetmp{n,jl}} \Big( \partial_t u_{\overlinetmp{n,jl}} + \nabla \cdot [m_1]_{\overlinetmp{n,jl}} \Big)\\
        &=
        \sum^{N_{t}-1}_{n=1} \sum^{N_{x_1}-1}_{j=0} \sum^{N_{x_2}-1}_{l=0} \Phi_{\overlinetmp{n,jl}} \Big( \frac{u_{\overlinetmp{n,jl}} - u_{\overlinetmp{n-1,jl}}}{\Delta t} + \frac{[m_{1,x_1}]_{n,j+1l}-[m_{1,x_1}]_{\overlinetmp{n,jl}}}{\Delta x_1} + \frac{[m_{1,x_2}]_{\overlinetmp{n,j,l+1}}-[m_{1,x_2}]_{\overlinetmp{n,jl}}}{\Delta x_2} \Big).
    \end{align*}
    By rearranging the indices,
    \begin{align*}
        =&
        - \sum^{N_{t}-2}_{n=1} \sum^{N_{x_1}-1}_{j=0} \sum^{N_{x_2}-1}_{l=0} \Big(\frac{\Phi_{\overlinetmp{n+1,jl}} - \Phi_{\overlinetmp{n,jl}}}{\Delta t}\Big)  u_{\overlinetmp{n,jl}}\\ 
        & - \sum^{N_{t}-1}_{n=1} \sum^{N_{x_1}-1}_{j=0} \sum^{N_{x_2}-1}_{l=0}  \Big( \frac{\Phi_{\overlinetmp{n,jl}}-\Phi_{\overlinetmp{n,j-1,l}}}{\Delta x_1}\Big) [m_{1,x_1}]_{\overlinetmp{n,jl}}\\
        & - \sum^{N_{t}-1}_{n=1} \sum^{N_{x_1}-1}_{j=0} \sum^{N_{x_2}-1}_{l=0}   \Big( \frac{\Phi_{\overlinetmp{n,jl}}-\Phi_{\overlinetmp{n,j,l-1}}}{\Delta x_2} \Big) [m_{1,x_2}]_{\overlinetmp{n,jl}}\\
        & + \sum^{N_{x_1}-1}_{j=1}\sum^{N_{x_2}-1}_{l=1}\frac{1}{\Delta t} (\Phi_{\overlinetmp{N_t-1,jl}} u_{\overlinetmp{N_t-1,jl}} - \Phi_{\overlinetmp{0,jl}} u_{\overlinetmp{0,jl}}).
    \end{align*}
    From above, we define partial differential operators of $\Phi$ as
    \[
        \begin{aligned}
            \partial_t \Phi_{\overlinetmp{n,jl}} &= (\Phi_{\overlinetmp{n+1,jl}} - \Phi_{\overlinetmp{n,jl}})/\Delta t,\\ 
            \partial_{x_1} \Phi_{\overlinetmp{n,jl}} &= (\Phi_{\overlinetmp{n,jl}} - \Phi_{\overlinetmp{n,j-1,l}})/\Delta x_1,\\ \partial_{x_2} \Phi_{\overlinetmp{n,jl}} &= (\Phi_{\overlinetmp{n,jl}} - \Phi_{\overlinetmp{n,j,l-1}})/\Delta x_2.
        \end{aligned}
    \]
    Putting back to the Lagrangian functional, we get
    \begin{equation*}
        \begin{aligned}
            &\mathcal{L}(m_1,m_2,u,\Phi)\\
            &\approx 
            \Delta t \Delta x_1 \Delta x_2 \sum^{N_{t}-1}_{n=1} \sum^{N_{x_1}-1}_{j=0} \sum^{N_{x_2}-1}_{l=0} \Bigg( \frac{([m_{1,x_1}]_{\overlinetmp{n,jl}})^2 + ([m_{1,x_1}]_{\overlinetmp{n,jl}})^2}{2 V_1(u_{\overlinetmp{n,jl}})}
            +  \frac{([m_{2}]_{\overlinetmp{n,jl}})^2}{2 V_2(u_{\overlinetmp{n,jl}})}
            - \Phi_{\overlinetmp{n,jl}} \, [m_2]_{\overlinetmp{n,jl}}\\
            & \hspace{6cm} - \partial_{x_1} \Phi_{\overlinetmp{n,jl}} [m_{1,x_1}]_{\overlinetmp{n,jl}} -  \partial_{x_2} \Phi_{\overlinetmp{n,jl}} [m_{1,x_2}]_{\overlinetmp{n,jl}} + s(u_{\overlinetmp{n,jl}})  \Bigg)\\
            &- \Delta t \Delta x_1 \Delta x_2 \sum^{N_{t}-2}_{n=1} \sum^{N_{x_1}-1}_{j=0} \sum^{N_{x_2}-1}_{l=0} \partial_t \Phi_{\overlinetmp{n,jl}}  u_{\overlinetmp{n,jl}}\\ 
        & + \Delta x_1 \Delta x_2 \sum^{N_{x_1}-1}_{j=1}\sum^{N_{x_2}-1}_{l=1} \Phi_{\overlinetmp{N_t-1,jl}} u_{\overlinetmp{N_t-1,jl}} - \Phi_{\overlinetmp{0,jl}} u_{\overlinetmp{0,jl}} +  G(u_{\overlinetmp{N_t-1,jl}}).
        \end{aligned}
    \end{equation*}
    By differentiating with respect to $[m_{1,x_i}]_{\overlinetmp{n,jl}}$ $(i=1,2)$, $[m_2]_{\overlinetmp{n,jl}}$, $u_{\overlinetmp{n,jl}}$, we achieve the optimality conditions~\eqref{eq:optimality_m1_discrete},~\eqref{eq:optimality_m2_discrete},~\eqref{eq:optimality_u_discrete}.

\end{proof}

From Proposition~\ref{prop:discrete_optimality}, the variational problem solves the following PDE:
\begin{equation}\label{eq:discrete_PDE}
    \begin{aligned}
        &\partial_t u + \nabla \cdot m_1 = m_2\\
        &\partial_t \Phi + \frac{1}{2}  V_1'(u) |\nabla \Phi|^2 + \frac{1}{2} V_2'(u) \Phi^2  = s'(u)\\
        &u(0,x) = u_0(x), \quad \Phi(1,x) = - G'(u(1,x)),\quad x\in\Omega.
    \end{aligned}    
\end{equation}

\noindent Note that solving this system of PDEs is computationally challenging. The PDE of $u$ evolves forward in time with an initial condition $u_0$ and that of $\Phi$ evolves backward in time with a terminal condition $-G'(u(1,x))$. Solving the PDEs using the finite difference methods is difficult due to the strict CFL condition (the ratio of $\Delta t$, $\Delta x_1$, $\Delta x_2$). Furthermore, using finite difference schemes that are implicit in time is also challenging because there is no simple formula of $u_{\overlinetmp{n,jl}}$ in terms of $u_{\overlinetmp{n-1,jl}}$, $[m_1]_{\overlinetmp{n,jl}}$, $[m_2]_{\overlinetmp{n,jl}}$ that satisfy the PDE of $\Phi$~\eqref{eq:discrete_PDE}. With nonlinear $V_1$ and $V_2$, it becomes even more difficult to use any finite difference schemes. 

Instead of solving PDEs directly using finite difference schemes, the algorithm computes the saddle point of a variational problem~\eqref{variation} that satisfies the optimality conditions in Proposition~\ref{prop:discrete_optimality}. Thus, the saddle points are in fact the solutions to the PDE~\eqref{eq:discrete_PDE}.
Furthermore, by~\eqref{eq:optimality_phi_discrete} and~\eqref{eq:optimality_u_discrete}, the algorithm solves~\eqref{eq:discrete_PDE} with implicit finite difference schemes (forward in time for $u$ and backward in time for $\Phi$). Thus, the algorithm circumvents the numerical difficulties coming from the strict CFL conditions.

\section{Numerical Examples}
In this section, we present two sets of numerical experiments using the Algorithm~2 with different $V_1$, $V_2$ functions. We wrote C++ codes to run the numerical experiments. For all the experiments we used the discretization sizes $N_{x_1} = N_{x_2} = 128$, $N_{t} = 30$ and $c=0.1$. Furthermore, in the numerical experiments, we use 
\begin{equation*}
 \mathcal{F}(u(t,\cdot))(x) =  -c \int_\Omega \Big(u(t,x) \log u(t,x) - u(t,x)\Big)\, dx,  
\end{equation*}
 with a given positive constant $c$. 

\subsubsection{Numerical Example 1}
The numerical simulations are based on the choice of 
$$V_1(u) = u \text{ and } V_2(u) = u^\alpha$$ 
where $\alpha>0$.  The choices are from Example~5 in Section~\ref{section4}. 
We show how varying the exponent $\alpha$ affects the evolution of the densities. Below are the initial and terminal densities used in the simulations.

\begin{equation*}\label{eq:exp1-initial-densities}
    \begin{aligned}
        u_0(0,x_1,x_2) &= 10 \exp\bigl(-60\left((x_1-0.3)^2+(x_2-0.7)^2\right)\bigr) + 1\\
        u_1(0,x_1,x_2) &= 20 \exp\bigl(-60\left((x_1-0.7)^2+(x_2-0.3)^2\right)\bigr) + 1.
    \end{aligned}
\end{equation*}
Furthermore, we use the terminal functional
\begin{equation}\label{eq:experiment1_terminal_functional}
    \mathcal{G}(u(1,\cdot)) = \int_\Omega i_{u_1(x)}(u(1,x))\, dx
\end{equation}
where 
\[
    i_{u_1(x)}(u(1,x)) = \begin{cases}
        0 & \text{if } u(1,x) = u_0(x)\\
        \infty & \text{otherwise.}
    \end{cases}
\]
This function is equivalent to the following discrete form:
\[
    i_{[u_1]_{\overlinetmp{jl}}}(u_{\overlinetmp{0,jl}}) = \begin{cases}
        0 & \text{if } u_{\overlinetmp{N_t-1,jl}} = [u_1]_{\overlinetmp{jl}}\\
        \infty & \text{otherwise}
    \end{cases}
\]
for $j=0,\cdots,N_{x_1}-1$ and  $l=0,\cdots,N_{x_2}-1$.
Recall that, from Proposition~\ref{prop:discrete_optimality}, the boundary condition of $\Phi$ at $t=1$ is given by
\[
    \Phi_{\overlinetmp{N_t-1,jl}} = - (i_{[u_1]_{\overlinetmp{jl}}})'(u_{\overlinetmp{N_t-1,jl}}).
\]
However, dealing with the derivative of $i_{[u_1]_{\overlinetmp{jl}}}$ is computationally challenging because it is non-differentiable and discontinuous. Since the function is convex, we may use a property Legendre transform, i.e.,
\[
    (f^*)' = (f')^{-1}
\]
for any convex function $f$.
Using the property, we get the boundary condition of $u$ at time $t=1$,
\[
    u_{\overlinetmp{N_t-1,jl}} = (i^*_{[u_1]_{\overlinetmp{jl}}})'( - \Phi_{\overlinetmp{N_t-1,jl}})
\]
where
\[
\begin{aligned}
    i^*_{[u_1]_{\overlinetmp{jl}}}(\Phi) &= \sup_{u \in \mathbb{R}} \quad u\cdot\Phi   - i_{[u_1]_{\overlinetmp{jl}}}(u) = [u_1]_{\overlinetmp{jl}} \Phi.
\end{aligned}
\]
Thus, we have
\[
    u_{\overlinetmp{N_t-1,jl}} = [u_1]_{\overlinetmp{jl}}.
\]
In other words, the terminal functional~\eqref{eq:experiment1_terminal_functional} is equivalent to the boundary condition of $u$ in time at $t=1$.

The variational problem~\eqref{variation} is solved with these initial conditions using Algorithm~2. The solution of the variational problem represents the evolution of the density $u$ from initial density $u_0$ at time $t=0$ to terminal density $u_1$ at $t=1$ that satisfies the system of PDEs~\eqref{eq:discrete_PDE}. The results are shown in Figure~\ref{fig:exp1-imshows}, Figure~\ref{fig:exp1-lines} and Figure~\ref{fig:exp1-conv}. In Figure~\ref{fig:exp1-imshows}, the plots show the evolution of densities from time $t=0$ to $t=1$. The first row is from $V_2(u)=u$, the second row is from $V_2(u) = u^2$, the third row is from $V_2(u) = u^3$ and the last row is from $V_2(u) = u^4$. The Figure~\ref{fig:exp1-lines} shows the cross sections of the 2d plots from Figure~\ref{fig:exp1-imshows} along a line $x+y=1$.

\begin{figure}[h!]
    \centering
    \includegraphics[width=0.97\linewidth]{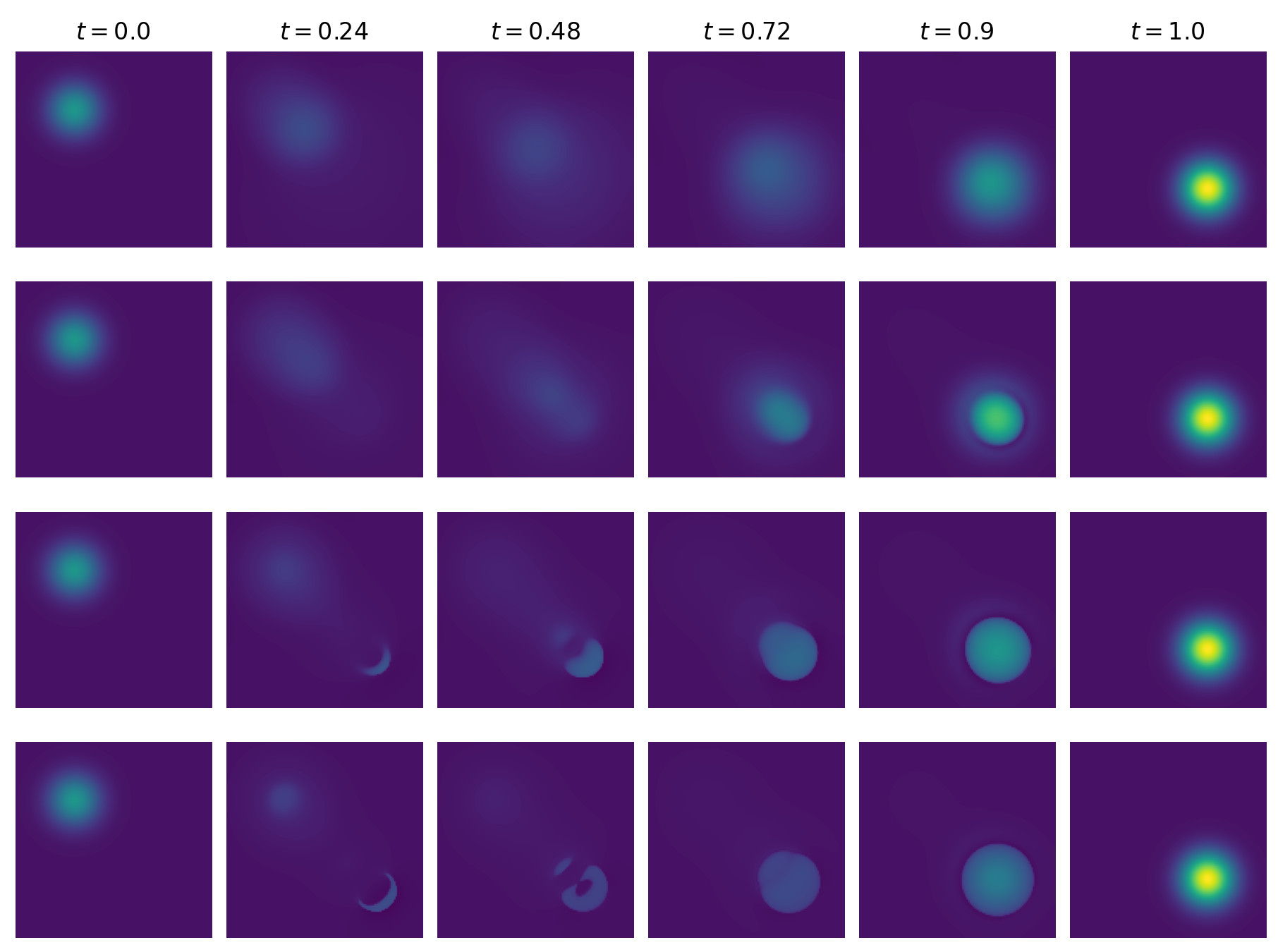}
    \caption{Example 1: Snapshots of the densities at different times $t$. Each column shows the results at $t=0$ (column 1), $t=0.24$ (column 2), $t=0.48$ (column 3), $t=0.72$ (column 4), $t=0.9$ (column 5), $t=1$ (column 6). The first, the second, the third and the fourth rows represent $V_2(u)=u$, $V_2(u)=u^2$, $V_2(u)=u^3$, $V_2(u)=u^4$, respectively. $V_1(u)=u$ for all simulations.}
    \label{fig:exp1-imshows}
\end{figure}

Note that the function $(u,m)\mapsto \frac{\|m\|_2^2}{2 V(u)}$ is nonconvex when $V(u)=u^\alpha$ and $\alpha>1$. Thus, the solutions for $V_2(u)=u^2, u^3, u^4$ of the variational problems are not unique and the numerical results may not represent the global minimizer of the formulation~\eqref{variation}. However, the algorithm converges to a local minimum and the convergence plot can be seen in Figure~\ref{fig:exp1-conv}. In Figure~\ref{fig:exp1-conv}, the vertical axis represents the value of the energy functional given in~\eqref{osher} and the horizontal axis represents iterations. 

\begin{figure}
    \centering
    \includegraphics[width=0.5\linewidth]{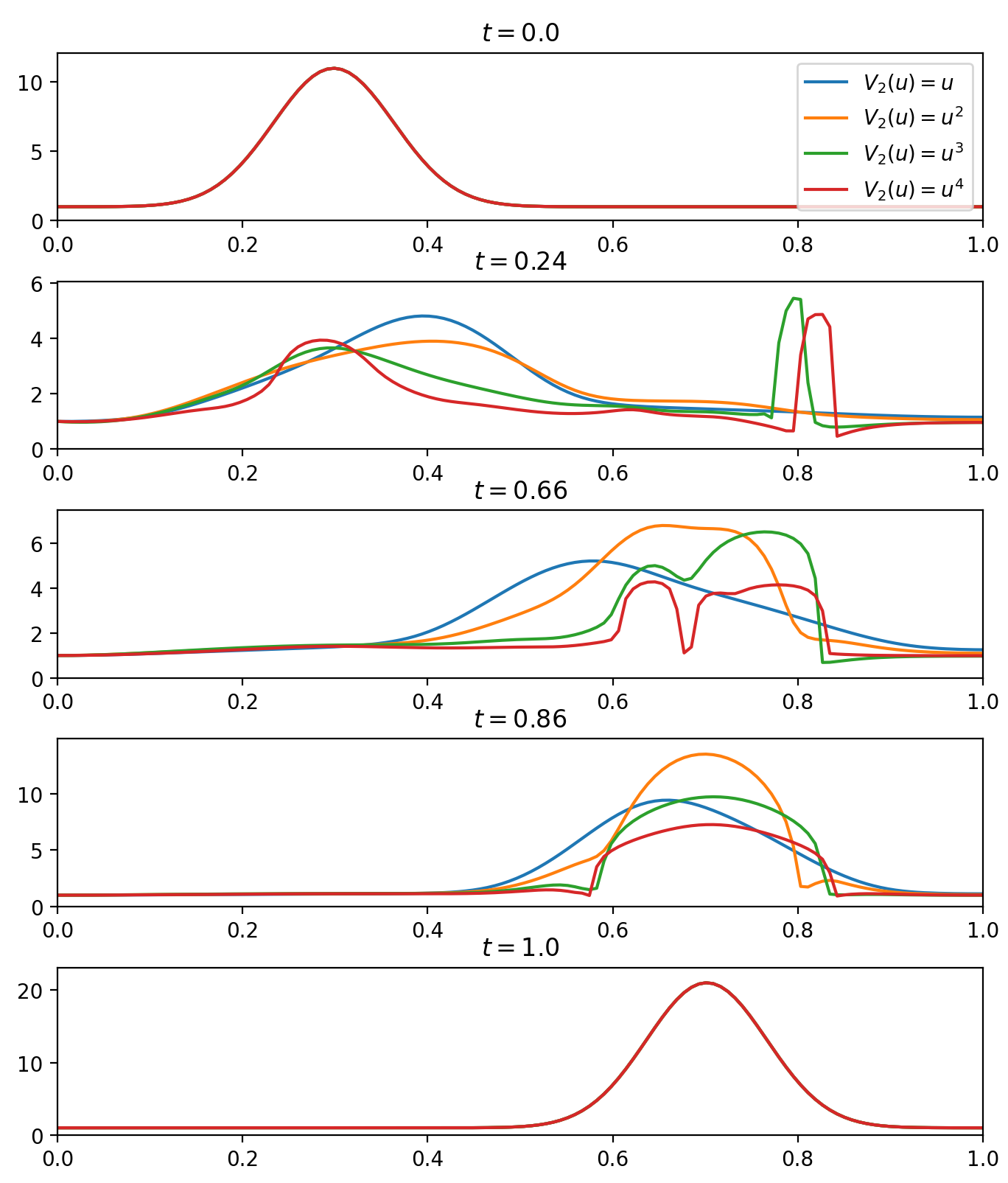}
    \caption{Example 1: Cross sections along a line $x+y=1$ for different time $t$. Each color represents $V_2(u)=u$, $V_2(u)=u^2$, $V_2(u)=u^4$.}
    \label{fig:exp1-lines}
\end{figure}
\begin{figure}
    \centering
    \includegraphics[width=0.6\linewidth]{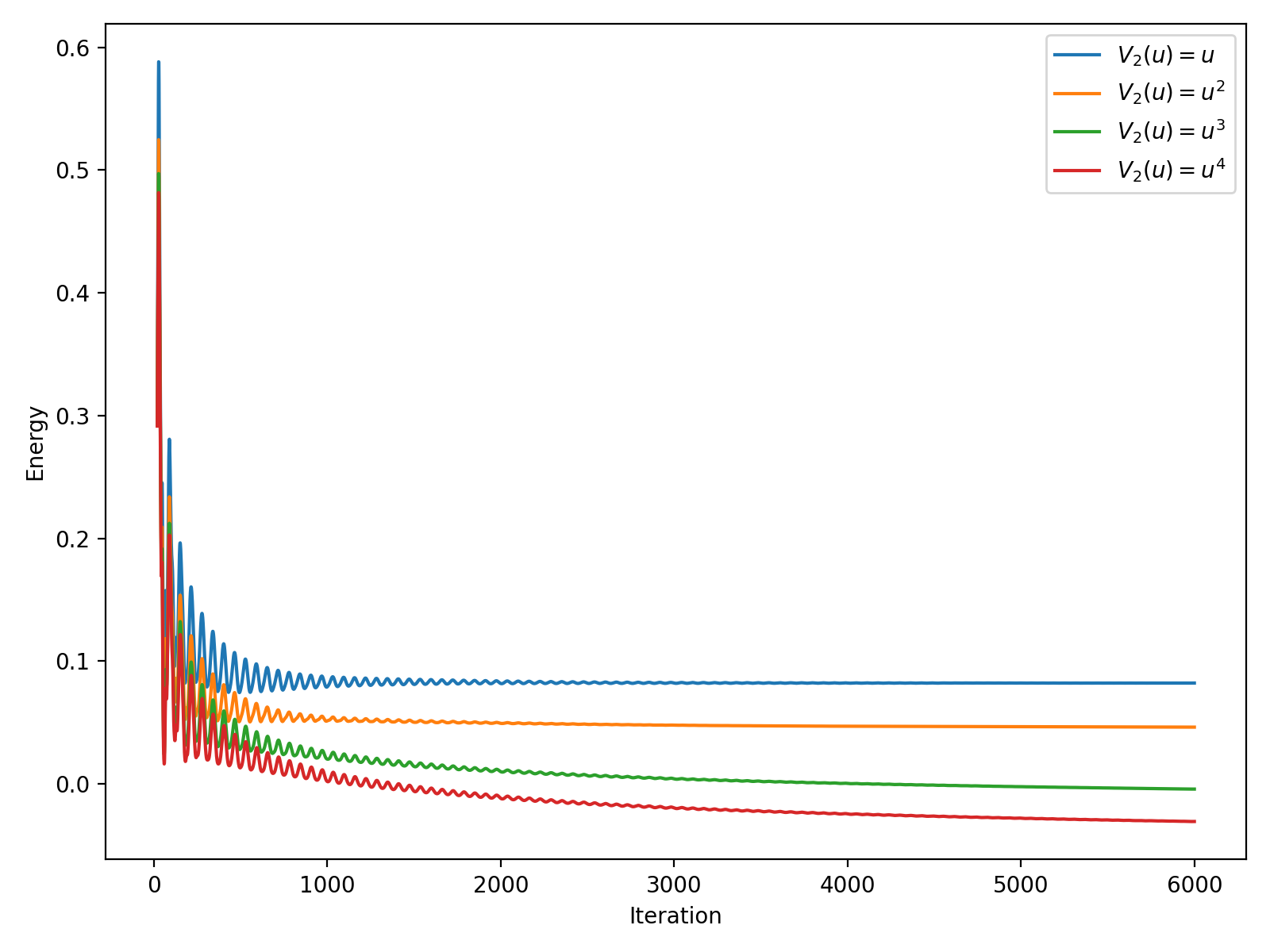}
    \caption{Example 1: The convergence plot from the Algorithm. The x-axis represents iterations and y-axis represents the energy defined in~\eqref{osher}.}
    \label{fig:exp1-conv}
\end{figure}
\newpage
\subsubsection{Numerical Example 2}
The numerical simulations are based on the choice of $V_2(u) = \frac{u(u-1)}{\log u}$ which is from the Fisher-KPP equation in Example \ref{FKPP}. 
In particular, this result compares the following three sets of functions.
\begin{table}[h]
    \centering
    \begin{tabular}{c|c|c}
        & $V_1(u)$ & $V_2(u)$ \\
        \hline
        1 & $u$ & $u(u-1)/\log u$ \\
        2 & $\sqrt{u}$ & $u(u-1)/\log u$ \\
        3 & $\sqrt{u}$ & $u$ \\
    \end{tabular}
\end{table} 

 Below are the initial and terminal densities used in the simulations.

\begin{equation*}
    \begin{aligned}
        u_0(0,x) &= 15 \exp\bigl(-80\left((x-0.5)^2+(y-0.5)^2\right)\bigr) + 1\\
        u_1(0,x) &= 15 \exp\bigl(-80\left((x-0.3)^2+(y-0.3)^2\right)\bigr) + 15 \exp\bigl(-80\left((x-0.7)^2+(y-0.7)^2\right)\bigr) + 1.
    \end{aligned}
\end{equation*}
As in example~1, the terminal functional is defined as
\[
    \mathcal{G}(u(1,\cdot)) = \int_\Omega i_{u_1(x)}(u(1,x))\, dx.
\]
The variational problem~\eqref{variation} with these initial conditions are solved using Algorithm~2. The results are shown in Figure~\ref{fig:exp2-imshows} and Figure~\ref{fig:exp2-conv}. {This example is also computed from our Algorithm 2. Figure \ref{fig:exp2-imshows} demonstrates the evolution of densities for different $V_1$ and $V_2$ for the time interval $[0,1]$. Figure \ref{fig:exp2-conv} shows the optimal control energy functional value for different choices of $V_1$, $V_2$. The evolution of density functions are similar, but the energy functionals are different.}

\begin{figure}[H]
    \centering
    \includegraphics[width=0.97\linewidth]{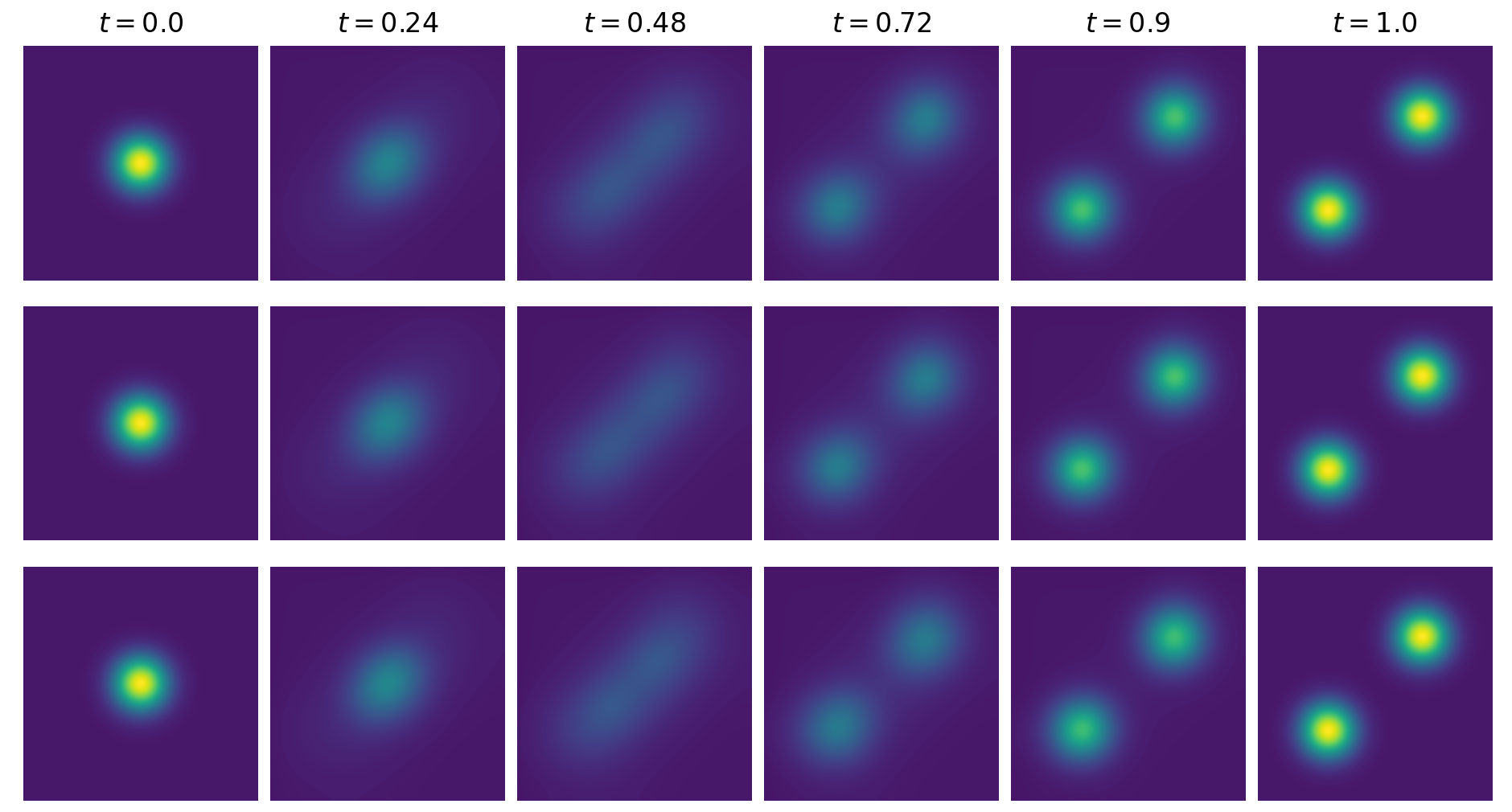}
    \caption{Example 2: Snapshots of the densities at different times $t$. Each column shows the results at $t=0$ (column 1), $t=0.24$ (column 2), $t=0.48$ (column 3), $t=0.72$ (column 4), $t=0.9$ (column 5), $t=1$ (column 6). The first row is from ($V_1(u)=u$, $V_2(u)=\frac{u(u-1)}{\log u}$), the second row is from ($V_1(u)=\sqrt{u}$, $V_2(u)=\frac{u(u-1)}{\log u}$) and the last row is from ($V_1(u)=\sqrt{u}$, $V_2(u)=u$).}
    \label{fig:exp2-imshows}
\end{figure}
\begin{figure}
    \centering
    \includegraphics[width=0.5\linewidth]{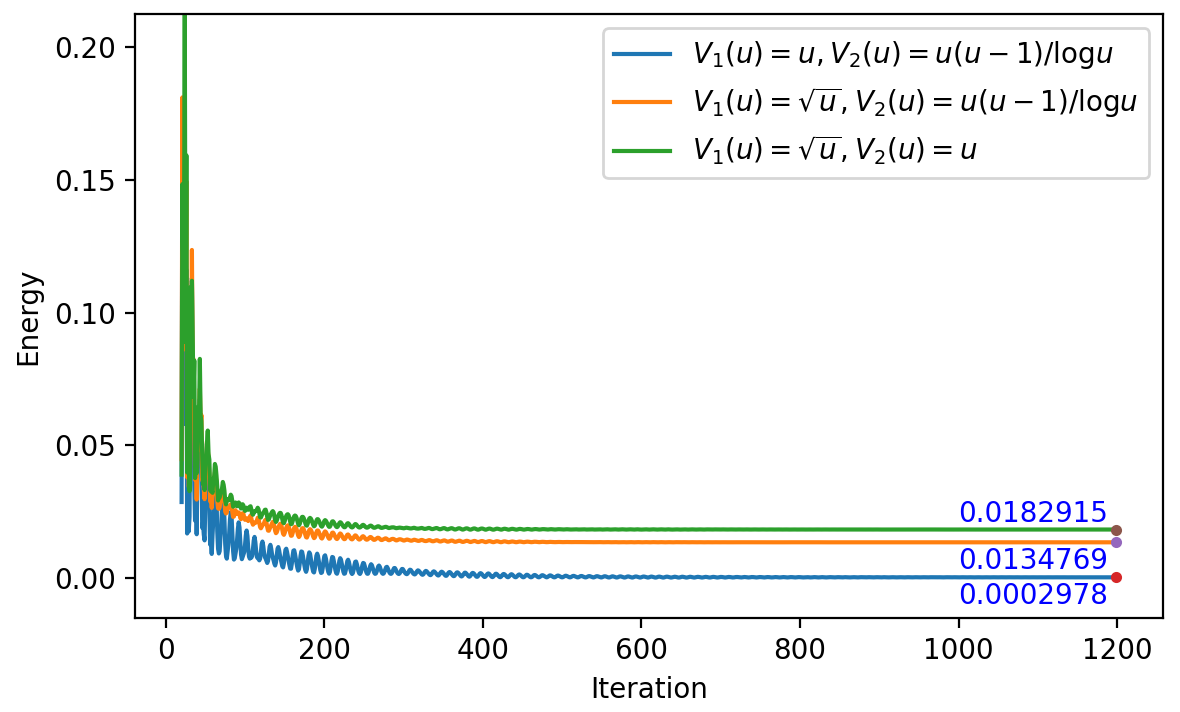}
    \caption{Example 2: The convergence plot from the Algorithm. The x-axis represents iterations and y-axis represents the energy defined in~\eqref{osher}. The blue line represents the solution of $V_1(u)=u$ and $V_2(u)=\frac{u(u-1)}{\log u}$ and converges to $0.00029$. The orange line represents the solution of $V_1(u)=\sqrt{u}$ and $V_2(u)=\frac{u(u-1)}{\log u}$ and converges to $0.0135$. The green line represents the solution of $V_1(u)=\sqrt{u}$ and $V_2(u)=u$ and converges to $0.0183$.}
    \label{fig:exp2-conv}
\end{figure}

\section{Discussion}
In this paper, we applied a novel generalized mean field control for nonlinear reaction-diffusion equations. Several mean-field information variational problems and dynamics in unnormalized density space are presented. In computations, we focus on mean field control problems. And we design primal-dual hybrid gradient methods to compute these generalized optimal transport and mean-field control problems. In future work, we shall design and compute unconditional stable implicit time schemes for reaction-diffusion equations using mean field control problems.

\bibliographystyle{abbrv}

\end{document}